\documentclass[11pt]{article}

\usepackage[utf8]{inputenc}
\usepackage[english]{babel}
\usepackage{enumerate}
\usepackage{latexsym}
\usepackage{amsthm,amsmath}
\usepackage{amssymb}
\usepackage{aliascnt} 
\usepackage{multicol}
\usepackage{tikz}
\usepackage{float} 

\usepackage[bookmarks]{hyperref} 
\hypersetup{
	pdftitle={Classification of prime graphs with 2-switch-degree at most 4},
	pdfauthor={Victor N. Schvollner},
	colorlinks=true, linkcolor=blue, citecolor=red,
	filecolor=cyan, urlcolor=magenta
}

\usetikzlibrary{babel,decorations.pathreplacing,calc,positioning,shapes.geometric,arrows.meta}

\usepackage{cleveref}

\definecolor{10}{RGB}{115,59,171}
\definecolor{8}{RGB}{212,122,240}
\definecolor{7}{RGB}{99,212,119}
\definecolor{6}{RGB}{183,240,164}
\definecolor{D}{RGB}{255,162,79}
\definecolor{E}{RGB}{255,84,0}
\definecolor{F}{RGB}{158,248,255}
\definecolor{G}{RGB}{128,135,255}
\definecolor{I}{RGB}{187,255,0}
\definecolor{A}{cmyk}{.9,.05,.4,0}
\definecolor{B}{RGB}{150,30,150}
\definecolor{C}{RGB}{186,155,189}
\definecolor{9}{RGB}{0,180,60}
\definecolor{0}{RGB}{30,123,191}
\definecolor{1}{RGB}{255,113,102}
\definecolor{2}{RGB}{41,199,92}
\definecolor{3}{RGB}{242,207,16}
\definecolor{5}{RGB}{255,15,154}
\definecolor{4}{rgb}{.8,0,.8}

\definecolor{Red}{rgb}{1,0.4,0.4}
\definecolor{Green}{rgb}{.1,.5,.1}
\definecolor{Blue}{rgb}{.1,.1,.5}
\definecolor{blue}{RGB}{0,0,255}
\definecolor{Yellow}{rgb}{.8,.4,0}
\definecolor{X}{rgb}{.8,.4,0}
\definecolor{H}{rgb}{0,0,1}
\definecolor{light}{rgb}{.67,.84,.90}
\definecolor{Cyan}{rgb}{0,1,1}
\definecolor{Purple}{rgb}{.5,0,.5}
\definecolor{Purple2}{rgb}{.5,.2,.5}
\definecolor{white}{rgb}{1.0,1.0,1.0}
\definecolor{Purple2}{rgb}{.8,.4,0}
\definecolor{Amarillo}{RGB}{225,191,73}
\definecolor{Celeste}{RGB}{117,170,219}
\definecolor{Castano}{RGB}{232,53,17}
\definecolor{Black}{RGB}{0,0,0}
\definecolor{White}{RGB}{255,255,255}
\definecolor{gris}{rgb}{.5,.5,.5}

\newtheorem{theorem}{Theorem}[section]

\newaliascnt{corollary}{theorem}
\newtheorem{corollary}[corollary]{Corollary}
\aliascntresetthe{corollary}

\newaliascnt{lemma}{theorem}
\newtheorem{lemma}[lemma]{Lemma}
\aliascntresetthe{lemma}

\newaliascnt{proposition}{theorem}
\newtheorem{proposition}[proposition]{Proposition}
\aliascntresetthe{proposition}

\pgfdeclarelayer{background2}
\pgfdeclarelayer{background}
\pgfdeclarelayer{foreground}
\pgfsetlayers{background2,background,main,foreground}

\title{Classification of prime graphs with \\ 2-switch-degree at most 4}

\author{Victor N. Schv\"ollner\thanks{%
		Instituto de Matem\'atica Aplicada San Luis (IMASL), UNSL--CONICET,
		San Luis, Argentina. \texttt{vnsi9m6@gmail.com}.
}}
\date{}

\begin{document}
	\maketitle

\begin{abstract}
			The 2-switch-degree $\deg(G)$ of a graph $G$ is the number of
		2-switches that can be performed on $G$; equivalently, it is the degree
		of $G$ as a vertex of the realization graph $\mathcal{G}(d)$ of its
		degree sequence $d$. We classify the prime graphs of 2-switch-degree at
		most~4, where a graph is \emph{prime} if it is indecomposable with
		respect to the Tyshkevich composition and every vertex takes part in
		some 2-switch. From this classification we derive a sharp dichotomy
		that recovers the global shape of $\mathcal{G}(d)$ from a single one of
		its local degrees: if $d$ has a realization $X$ with $\deg(X)=k\le3$,
		then $\mathcal{G}(d)$ is vertex-transitive and $k$-regular if and only
		if neither $T_{221}$ nor $\overline{T_{221}}$ is an induced subgraph
		of $X$. Moreover, for every $k\geq 4$ some prime
		graph of degree~$k$ carries a 2-switch raising its degree to $2k-2$.
		As a further consequence, up to isomorphism there are only
		10 realization graphs of prime graphs with degree at most~4.
	
	
\end{abstract}

\section{Introduction}

The systematic study of the graphs sharing a common degree sequence has
long been a central theme in graph theory. A fundamental tool in this
context is the \emph{2-switch}: given two vertex-disjoint edges $ab$ and
$xy$ of a graph $G$ with $ax,by\notin E(G)$, the 2-switch on
$\{ab,xy\}$ replaces these two edges by $ax$ and $by$
(see \cite{chartrand2010graphs}). 
Since it preserves degrees, it moves between the
realizations of a fixed degree sequence $d$, which form the vertices of
the \emph{realization graph} $\mathcal{G}(d)$: the labeled graphs on a
fixed vertex set with degree sequence $d$, two being adjacent whenever
one is obtained from the other by a single 2-switch
(see \cite{arikati1999realization}). Classically, any two realizations of
$d$ are related by a finite sequence of 2-switches (see page 24 of
\cite{chartrand2010graphs}); equivalently, $\mathcal{G}(d)$ is connected.

Understanding the structure of $\mathcal{G}(d)$ is a rich area of
research: Arikati and Peled proved in \cite{arikati1999realization} that $\mathcal{G}(d)$ is Hamiltonian
whenever $d$ has majorization gap~1,
Taylor studied in \cite{taylor2006contrained} constrained versions of the switching operation, and more recent work has focused on the
connectivity of the subgraphs of $\mathcal{G}(d)$ induced by particular
graph families (see  \cite{schvollner_pseudoforests}).
A more granular, local description of $\mathcal{G}(d)$ is provided by
the \emph{2-switch-degree} (or simply the \emph{degree}) of a graph $G$,
denoted $\deg(G)$, which is the number of distinct 2-switches that can
be performed on $G$. Since two distinct 2-switches on $G$ produce two
distinct graphs, $\deg(G)$ is precisely the degree of $G$ as a vertex of
$\mathcal{G}(d)$. This parameter was introduced
in \cite{pastine20252}, where it is expressed in terms of the induced
subgraphs of $G$ of order~4.

The study of 2-switch transformations is closely related to the
Tyshkevich canonical decomposition \cite{tyshkevich2000decomposition}, which expresses every graph as a
composition $G=G_n\circ\cdots\circ G_1$ of indecomposable graphs, in a
way that is unique up to isomorphism when all the factors have at least one vertex. A key feature of this decomposition
is that all its factors, with the possible exception of $G_1$, are split
graphs. A graph $S$ is \emph{split} if its vertex set can be
partitioned into a clique $K$ and an independent set $I$; when such a
partition is fixed we write
$(S,K,I)$ and we say that $(K,I)$ is a \emph{bipartition} for $S$. Split graphs arise naturally in several areas of graph theory
and have been widely studied
(see \cite{hammer1977split, splitnordhausgaddum, whitman2020split,jaume2025nullspace}). For split graphs the 2-switch-degree admits a
particularly simple description: it equals the number of induced
$P_4$'s.

A vertex of $G$ is \emph{active} if it takes part in some 2-switch on $G$. A graph whose vertices are all active is itself called \emph{active}, and an indecomposable active graph is called \emph{prime}
(see \cite{pastine20252}). Since every factor in the Tyshkevich decomposition
of an active graph is prime, and since all these factors but at most one
are split, prime split graphs are the basic building blocks of the whole
theory. Their structure is further clarified by the \emph{factor graph}
$\Phi=\Phi(S,K,I)$, a loopless multigraph on the vertex set $I$ that records the
2-switches available in $S$. This notion was introduced
in \cite{pastine2025simple}, inspired by the $A_4$-structure of Barrus
and West \cite{barrus.west.A4}; there it is shown that an active split
graph $S$ is prime if and only if $\Phi$ is connected.

Building on these foundations, we classify the graphs whose
2-switch-degree is small, that is, those that are least flexible under
2-switch transformations. Our starting point is a complete classification
of the active graphs of degree at most~3
(\Cref{clasificacion.activo_deg=2,clasificacion.activo_deg=3}); the prime
ones among them are all split except $C_4$ and $2K_2$. This already
determines the global shape of a realization graph from a single one of
its local degrees: if $X$ realizes $d$ and $\deg(X)=k\le3$, then
$\mathcal{G}(d)$ is vertex-transitive and $k$-regular if and only if
$T_{221}$ and $\overline{T_{221}}$ are not induced subgraphs of $X$
(\Cref{G(d).transitivo}), the only exception up to complementation. The
hypothesis $k\le3$ cannot be dropped: for every $k\ge 4$ there is a prime
graph $G_k$ of degree $k$ carrying a 2-switch $\tau$ with
$\deg(\tau(G_k))=2k-2$ (\Cref{familia.Gk}), so from degree~4 on a single
2-switch may almost double the degree and regularity fails. Degree~4 is
thus the first value at which the picture is no longer rigid, and it is
the one we settle here: there are exactly 18 prime graphs of degree~4
(\Cref{clasificacion.primos.deg=4}), 12 split
(\Cref{clasificacion.split.primos.deg=4}) and 6 non-split
(\Cref{clasificacion.no.split.primos.deg=4}). Finally, we determine the
realization graphs of all 29 prime graphs of degree at most~4
(\Cref{realization.graphs.deg<5}): up to isomorphism there are only 10, so
a short list accounts for every realization graph arising in this range.
Some of our tools are stated beyond degree~4; for instance,
\Cref{C4.distinguidores} applies to every graph of degree at most~5.

The rest of the article is organized as follows. \Cref{sec_tools}
gathers the tools used throughout the paper. In \Cref{sec_deg<3} we
classify the active graphs of degree at most~2, and in \Cref{sec_deg=3}
those of degree~3. \Cref{sec_deg=4,sec_deg=4.no.split} are devoted to the
prime graphs of degree~4, split and non-split respectively. \Cref{sec_realization} studies the regularity of
realization graphs having a vertex of degree $k\le4$ and classifies them.


\section{Tools}\label{sec_tools}	

Let $G$ be a graph. We write $V(G)$ and $E(G)$ for its vertex set and
its edge set, and $|G|=|V(G)|$ for its order. The complement of $G$ is
denoted by $\overline{G}$. We write $H\prec G$ to mean that $H$ is an
induced subgraph of $G$, and $G[W]$ (or $G[v_1,\dots,v_k]$) for the
subgraph of $G$ induced by $W\subseteq V(G)$. The notation $G\approx H$
means that $G$ and $H$ are isomorphic. For $n\in\mathbb{N}=\{1,2,\ldots\}$ we put
$[n]=\{1,\dots,n\}$, and $\dot\cup$ denotes disjoint union.
The \emph{neighborhood} of $v\in V(G)$ is
$N_G(v)=\{x\in V(G):vx\in E(G)\}$, and $\deg_G(v)=|N_G(v)|$ is the
\emph{degree} of $v$ in $G$. When $G$ is clear from the context we
write $N_v$ and $d_v$. A path is denoted by the juxtaposition of its
vertices, so $uxyv$ is the $P_4$ with edges $ux,xy,yv$; for a cycle, the first vertex is repeated at the end. We write
$\omega(G)$ and $\alpha(G)$ for the clique number and the independence
number of $G$. 
If $V(G)=[n]$, the \emph{degree sequence} of $G$ is
$d=d(G)=(d_1,\dots,d_n)$. If $d=(d_v)_{v=1}^n$ is the degree
sequence of $G$, then $\overline{d}=(n-1-d_v)_{v=1}^n$ is the degree
sequence of $\overline{G}$. For graph-theoretical concepts not defined here we refer the reader to
\cite{chartrand2010graphs}.

For graphs $G$ and $H$, we denote by $\mathcal{Q}_G(H)$ the set of induced
subgraphs of $G$ isomorphic to $H$, and we put $\mathcal{Q}_G=\mathcal{Q}_G(P_4)\,\dot\cup\,\mathcal{Q}_G(C_4)\,\dot\cup\,\mathcal{Q}_G(2K_2)$.
If $\tau$ is a 2-switch on $G$, we write $\tau(G)$ for the resulting
graph. The four vertices on which $\tau$ acts induce a member of
$\mathcal{Q}_G$, and the next result counts the 2-switches on $G$ accordingly.

\begin{theorem}[\cite{pastine20252}]
	\label{degreeofG}
	For every graph $G$,
	\begin{equation}
		\label{eq18}
		\deg(G)=2|\mathcal{Q}_G (2K_2 )|+2|\mathcal{Q}_G (C_4 )|+|\mathcal{Q}_G (P_4 )|.
	\end{equation}
	In particular, if $G$ is a split graph, then $\deg(G)=|\mathcal{Q}_G (P_4 )|$.
\end{theorem}

By \eqref{eq18}, $\deg(P_4)=1$ and $\deg(C_4)=\deg(2K_2)=2$, so
\eqref{eq18} can be rewritten as
\begin{equation}
	\label{deg.como.suma}
	\deg(G)=\sum_{H\in \mathcal{Q}_G}\deg(H).
\end{equation}

%
%

\begin{theorem}
	\label{|G^*|}
	If $G$ is an active graph, then $4 \leq |G| \leq 4\deg(G)$.
\end{theorem}

\begin{proof}
	Since $G$ is active, it is clear that $\mathcal{Q}_G \neq \varnothing$, as $\deg(G) \geq 1$. Thus, $|G| \geq 4$. On the other hand:
	\begin{equation*}
		|G| = \left|\bigcup_{H \in \mathcal{Q}_G} V(H)\right| \leq \sum_{H \in \mathcal{Q}_G} 4 \leq 4 \sum_{H \in \mathcal{Q}_G} \deg(H) = 4\deg(G),
	\end{equation*}
	the last equality by \eqref{deg.como.suma}.
\end{proof}



\subsection{Split graphs and factor graphs}

If $(S,K,I)$ is a split graph and $G$ is a graph disjoint from $S$, the
\emph{Tyshkevich composition} $(S,K,I)\circ G$ is the graph with vertex
set $V(S)\cup V(G)$ and edge set
$E(S)\cup E(G)\cup\{xy:x\in K,\ y\in V(G)\}$
(see \cite{tyshkevich2000decomposition}); we denote by $S^n$
the composition of $n$ disjoint copies of a prime split graph $S$. A graph $G$ is
\emph{decomposable} if $G=S\circ H$ for some graphs $S$ and $H$ with $|S|,|H|\geq 1$, and \emph{indecomposable} otherwise.

The \emph{factor graph} of a split graph $(S,K,I)$ is the loopless
multigraph $\Phi=\Phi(S,K,I)$ with $V(\Phi)=I$ having one edge between $u$
and $v$ for each induced $P_4$ of $S$ containing both $u$ and $v$
\cite{pastine2025simple}. The multiplicity of the edge $uv$ is denoted by
$\sigma_{S}(uv)$, or simply $\sigma_{uv}$. As shown in
\cite{pastine2025simple}, the multiset $E(\Phi)$ does not depend on the bipartition, so the notation $\Phi(S)$ is unambiguous when $S$ is active, and $\deg(S)=|\mathcal{Q}_S(P_4)|=|E(\Phi(S))|$. 

\begin{theorem}[\cite{pastine2025simple}]
	\label{S.primo.iff.Phi(S).conexo}
	If $S$ is an active split graph, then $S$ is prime if and only if $\Phi(S)$ is connected. 
%
\end{theorem}

\begin{proposition}
	\label{|I|_|K|<=2deg(S)}
	If $(S,K,I)$ is a prime split graph, then 
	\[2 \leq |I|, |K| \leq \deg(S)+1.\]
%
\end{proposition}

\begin{proof}
	Every member of $\mathcal{Q}_S(P_4)$ meets $I$ in exactly two vertices, and likewise $K$ (see \cite{pastine2025simple}). If $|I| \leq 1$ or $|K| \leq 1$, then $\mathcal{Q}_S(P_4) = \varnothing$, so $\deg(S)=0$ by \Cref{degreeofG} and no vertex of $S$ is active. This gives $2 \leq |I|,|K|$.
		
		By \Cref{S.primo.iff.Phi(S).conexo}, $\Phi(S)$ is connected. Since $V(\Phi(S))=I$, $|E(\Phi(S))|=\deg(S)$, and a connected multigraph of size $m$ has order at most $m+1$, we obtain $|I| \leq \deg(S)+1$. Now $(\overline{S},I,K)$ is a prime split graph with clique $I$ and independent set $K$ (see page 14 of \cite{tyshkevich2000decomposition} and page 5 of \cite{pastine20252}). Thus, the previous argument applied to it yields $|K| \leq \deg(\overline{S})+1 = \deg(S)+1$ (see \cite{pastine20252} for the last equality). 
\end{proof}

\begin{lemma}
	\label{5>degG=|{P_4}|.implica.G.split}
	If $G$ is a graph such that $\deg(G)=|\mathcal{Q}_G(P_4)| \leq 4$, then $G$ is split.
\end{lemma}

\begin{proof}
	It is a well-known fact that a graph is split if and only if it contains no induced subgraphs isomorphic to $C_4$, $2K_2$, or $C_5$ (see \cite{hammer1977split}). The hypotheses imply that $G$ contains no induced subgraphs isomorphic to $C_4$ or $2K_2$. If $C_5 \prec G$, then $4 \geq \deg(G) \geq \deg(C_5)=5$.
	Thus, $G$ cannot contain induced subgraphs isomorphic to $C_5$ either.
\end{proof}



A split graph $S$ is indecomposable if and only if $\overline{S}$ is (see
\cite{tyshkevich2000decomposition}, page 14); a vertex is active in $S$ if
and only if it is active in $\overline{S}$ (see page 5 of \cite{pastine20252}),
and $\deg(\overline{S})=\deg(S)$ (see Corollary 6.2 of \cite{pastine20252});
and if $(K,I)$ is the bipartition of $S$, then $(I,K)$ is that of
$\overline{S}$, so $\omega(\overline{S})=\alpha(S)$ and
$\alpha(\overline{S})=\omega(S)$. Hence, to determine the prime split graphs
of degree $k$, it suffices to find those with $\alpha(S)\leq\omega(S)$, i.e.,
$|I|\leq|K|$, and add the complements of those with $\alpha(S)<\omega(S)$.
Those with $\alpha(S)=\omega(S)$ need not be self-complementary: for
instance, $R_{211}$ satisfies $\alpha=\omega=3$ and
$\overline{R_{211}}\not\approx R_{211}$ (see \Cref{split.primos.deg=4.fig.3}).

The following proposition lists some basic properties of $\sigma_{uv}$.

\begin{proposition}[\cite{pastine2025simple}]
	\label{prop.basicas.sigma_uv}
	Let $S$ be a split graph.
	\begin{enumerate}
		\item $\sigma_{uv}=(d_u-\eta_{uv})(d_v-\eta_{uv})$, where $\eta_{uv}=|N_u\cap N_v|$;
		
		\item $\sigma_{uv}=0$ and $d_v \leq d_u$ if and only if $N_v \subseteq N_u$; 
		
		\item $\sigma_{uv}=0$ and $d_u =d_v$ if and only if $N_u = N_v$;
		
		
		\item if $\sigma_{uv}=1$, then $d_u =d_v$ and $|N_u -N_v |=|N_v -N_u |=1$;
		
		\item if $d_u \geq d_v$ and $\sigma_{uv}=p$ is prime, then $|N_u -N_v |=p, |N_v -N_u |=1$ and $d_u-d_v=p-1$;
		
		
	\end{enumerate}
\end{proposition}

From now on, keep in mind that isolated and universal vertices of a graph $G$ are not active vertices of $G$ (see \cite{pastine20252}).

\begin{lemma}
	\label{S_activo_|I|=2}
	Let $(S,K,I)$ be an active split graph with $I=\{a,b\}$. Then: 
	\begin{enumerate}
		\item $|K|=d_a+d_b$ and $\sigma_{ab}=d_ad_b$;
		
		\item $S$ is isomorphic to the prime split graph with bipartition $([|K|],I)$ such that $N_a=[d_a]$ and $N_b=\{d_a+1,\ldots,|K|\}$, which we call $D_{d_b,d_a}$.
	\end{enumerate}
\end{lemma}

\begin{proof}
	\begin{enumerate}
		\item Since $S$ has no universal vertices, we have $\eta_{ab}=0$ and so $|K|=d_a+d_b$. Moreover, $\sigma_{ab}=d_ad_b$ by \Cref{prop.basicas.sigma_uv}.
		
		\item Straightforward. \qedhere
	\end{enumerate}
\end{proof}

\begin{lemma}
	\label{S_activo_|I|=3}
	Let $(S,K,I)$ be an active split graph with $I=\{a,b,c\}$. For $u\in I$, let $\pi_u$ be the number of vertices of $K$ whose only neighbor in $I$ is $u$, and let $\kappa_u$ be the number of vertices of $K$ whose neighbors in $I$ are exactly the two vertices of $I-\{u\}$. Then:
	\begin{enumerate}
		\item $|K|=\sum_{u\in I}(\pi_u+\kappa_u)$ and $d_u=\pi_u+\sum_{v\in I-\{u\}}\kappa_v$ for all $u\in I$;
		\item $\sigma_{uv}=(\pi_u+\kappa_v)(\pi_v+\kappa_u)$ for all distinct $u,v\in I$;
		\item $S$ is determined, up to isomorphism, by $(\pi_a,\pi_b,\pi_c,\kappa_a,\kappa_b,\kappa_c)$;
		\item if $\sigma_{ac}=0$, then, up to exchanging $a$ and $c$, we have $\pi_a=\kappa_c=0$, $\sigma_{ab}=\kappa_b(\pi_b+\kappa_a)$, $\sigma_{bc}=\pi_b(\pi_c+\kappa_b)$ and $|K|=\pi_b+\pi_c+\kappa_a+\kappa_b$.
	\end{enumerate}
\end{lemma}

\begin{proof}
	Since $S$ is active, every $x\in K$ is an inner vertex of an induced $P_4$ of $S$, whose ends lie in $I$ (see \cite{pastine2025simple}), so $N_x\cap I\neq\varnothing$; moreover, $x$ is not universal, so $N_x\cap I\neq I$. Hence the six sets counted by the $\pi_u$ and the $\kappa_u$ partition $K$, which gives (1), and $S$ is recovered, up to isomorphism, from their sizes, which gives (3). Since $|N_u-N_v|=\pi_u+\kappa_v$, (2) follows from \Cref{prop.basicas.sigma_uv}. Finally, if $\sigma_{ac}=0$, then by (2) we may assume, exchanging $a$ and $c$ if necessary, that $\pi_a+\kappa_c=0$, and (4) follows from (1) and (2).
\end{proof}

Following \cite{pastine2025simple}, we say that a split graph $(S,K,I)$ is \emph{$\delta$-homogeneous} if $\deg_S(v)=\delta$ for all $v\in I$. Recall that if $S$ is active, then $|K|=\omega(S)$, $|I|=\alpha(S)$, and $(K,I)$ is the unique bipartition for $S$. Throughout the next sections we use the following facts.

\begin{proposition}[\cite{pastine20252,pastine2025simple}]
	\label{hechos.previos}
	Let $(S,K,I)$ be a split graph and let $\Phi=\Phi(S,K,I)$.
	\begin{enumerate}
		\item If $G$ is a graph, then
		\[ \deg((S,K,I)\circ G)=\deg(S)+\deg(G), \]
		and $S\circ G$ is active if and only if $S$ and $G$ are active.
		
		\item If $K=\bigcup_{v\in I}N_v$, $|I|\geq 2$ and $\Phi$ is simple and connected, then $(S,K,I)$ is $\delta$-homogeneous.
		
		\item If $(S,K,I)$ is $\delta$-homogeneous and $\deg(S)\geq 1$, then:
		\begin{enumerate}
			\item $\sigma_{uv}=0$ if and only if $N_u=N_v$;
			\item $\Phi$ contains no three vertices $u,v,w$ such that $\sigma_{uv}\neq 0$ and $\sigma_{uw}=0=\sigma_{vw}$.
		\end{enumerate}
	\end{enumerate}
\end{proposition}


\subsection{Modules and induced 4-cycles}\label{sec_tools_C4}

A set $M\subseteq V(G)$ is a \emph{module} of $G$ if every vertex of
$V(G)- M$ is adjacent either to all vertices of $M$ or to none of
them (see \cite{gallai1967}). If $M$ is a module of $G$ and $W\subseteq V(G)$, then $M\cap W$ is a
module of $G[W]$.

\begin{lemma}
	\label{modulo.reemplazo}
	Let $M$ be a module of $G$, let $H\prec G$ and put $W=V(H)\cap M$. If
	$A\subseteq M$ and $\varphi\colon W\to A$ is an isomorphism from
	$G[W]$ onto $G[A]$, then $G\bigl[(V(H)-W)\cup A\bigr]\approx H$.
\end{lemma}

\begin{proof}
	Extend $\varphi$ by the identity on $V(H)-W$. Adjacencies inside
	$V(H)-W$ are unchanged and those inside $W$ are preserved by
	$\varphi$. Since $V(H)-W\subseteq V(G)-M$ and $M$ is a module, every
	$v\in V(H)-W$ is adjacent to all or to none of the vertices of $M$;
	hence $uv\in E(G)$ if and only if $\varphi(u)v\in E(G)$, for all
	$u\in W$.
\end{proof}

Let $A_4(G)$ be the graph on $V(G)$ in which two vertices are adjacent if
and only if they lie in a common member of $\mathcal{Q}_G$, that is, if some
2-switch on $G$ involves both of them; $A_4(G)$ is the $2$-section of the
$A_4$-structure of $G$ introduced in \cite{barrus.west.A4}.

\begin{theorem}[\cite{barrus.west.A4}; see also \cite{pastine2025simple}]
	\label{indecomp.characterization}
	If $G_n\circ\cdots\circ G_1$ is the Tyshkevich decomposition of $G$,
	then $A_4(G)=\dot{\bigcup}_{i=1}^{n}A_4(G_i)$. In particular, $G$ is
	indecomposable if and only if $A_4(G)$ is connected.
\end{theorem}

Note that \eqref{deg.como.suma} gives
$\sum_{F\in\mathcal{F}}\deg(F)\leq\deg(G)$ for every
$\mathcal{F}\subseteq\mathcal{Q}_G$, with equality if and only if
$\mathcal{F}=\mathcal{Q}_G$. If $C\in\mathcal{Q}_G(C_4)$ and $v\in V(G)-V(C)$, define
\[
\mathcal{H}_v=\bigl\{H\in \mathcal{Q}_G:\ V(H)\subseteq\{v\}\cup V(C)
\ \text{ and }\ V(H)\neq V(C)\bigr\}.
\]

Let $M\subseteq V(G)$. A vertex $v\in V(G)-M$ \emph{distinguishes} $M$ if $\varnothing\neq N_G(v)\cap M\neq M$; thus $M$ is a module of $G$ if and only if no vertex distinguishes it.

\begin{lemma}
	\label{conteo.local.C4}
	Let $C=c_0c_1c_2c_3c_0$ be an induced $C_4$ in $G$, with subscripts
	in $\mathbb{Z}_4$, and let $v\in V(G)-V(C)$. Then
	\[
	\sum_{H\in\mathcal{H}_v}\deg(H)=
	\begin{cases}
		0, & \text{if } v \text{ does not distinguish } V(C);\\[2pt]
		4, & \text{if } N_G(v)\cap V(C)=\{c_i,c_{i+2}\}
		\text{ for some } i\in\mathbb{Z}_4;\\[2pt]
		2, & \text{otherwise.}
	\end{cases}
	\]
\end{lemma}

\begin{proof}
	The $4$-subsets of $\{v\}\cup V(C)$ other than $V(C)$ are exactly the
	sets $W_i=\{v,c_{i-1},c_i,c_{i+1}\}$ with $i\in\mathbb{Z}_4$, so
	$\mathcal{H}_v=\{G[W_i]:i\in\mathbb{Z}_4\}\cap \mathcal{Q}_G$. Put
		$N=N_G(v)\cap V(C)$. Since $c_{i-1}c_i,c_ic_{i+1}\in E(G)$ and
	$c_{i-1}c_{i+1}\notin E(G)$, checking the eight possibilities for
	$N\cap\{c_{i-1},c_i,c_{i+1}\}$ gives
	\[
	G[W_i]\in \mathcal{Q}_G\iff c_i\notin N\ \text{ and }\
	N\cap\{c_{i-1},c_{i+1}\}\neq\varnothing,
	\]
	and in that case $\deg(G[W_i])=|N\cap\{c_{i-1},c_{i+1}\}|$, since
	$G[W_i]\approx P_4$ when this number is $1$ and $G[W_i]\approx C_4$
	when it is $2$.
	
	If $v$ does not distinguish $V(C)$, then $N\in\{\varnothing,V(C)\}$,
	no $W_i$ satisfies the condition, and the sum is $0$. Otherwise,
	rotating $C$ we may assume $N\in\{\{c_0\},\{c_0,c_1\},\{c_0,c_2\},
	\{c_0,c_1,c_2\}\}$. If $N=\{c_0,c_2\}$, then
	$\mathcal{H}_v=\{G[W_1],G[W_3]\}$ with both graphs isomorphic to
	$C_4$, so the sum is $4$. In each of the other three cases the sum is
	$2$: for $N=\{c_0\}$ we get $\mathcal{H}_v=\{G[W_1],G[W_3]\}$ with both
	graphs isomorphic to $P_4$; for $N=\{c_0,c_1\}$ we get
	$\mathcal{H}_v=\{G[W_2],G[W_3]\}$, again two copies of $P_4$; and for
	$N=\{c_0,c_1,c_2\}$ we get $\mathcal{H}_v=\{G[W_3]\}$ with
	$G[W_3]\approx C_4$.
\end{proof}

\begin{theorem}
	\label{C4.distinguidores}
	Let $C\prec G$ with $C\approx C_4$. Then
	\begin{equation}
		\label{conteo.global.C4}
		\deg(G)\ \geq\ 2+\sum_{v\in V(G)-V(C)}\ \sum_{H\in\mathcal{H}_v}\deg(H),
	\end{equation}
	with equality if and only if
	$\mathcal{Q}_G=\{C\}\cup\bigcup_{v\in V(G)-V(C)}\mathcal{H}_v$. Consequently,
	at most $\lfloor(\deg(G)-2)/2\rfloor$ vertices distinguish $V(C)$. In
	particular, $V(C)$ is a module of $G$ if $\deg(G)\leq 3$; and if
	$\deg(G)\leq 5$, then at most one vertex $v$ distinguishes $V(C)$,
	and $N_G(v)\cap V(C)$ is not a pair of opposite vertices of $C$.
\end{theorem}

\begin{proof}
	Every member of $\mathcal{H}_v$ contains $v$ and exactly three
	vertices of $C$; hence the families $\mathcal{H}_v$ are pairwise
	disjoint and none of them contains $C$. Since $\deg(C)=2$,
	\eqref{conteo.global.C4} and its equality case follow from
	\eqref{deg.como.suma}.
	
	By \Cref{conteo.local.C4}, a vertex contributes $0$ to the sum in
	\eqref{conteo.global.C4} if it does not distinguish $V(C)$, and at
	least $2$ if it does. If $t$ vertices distinguish $V(C)$, then
	\eqref{conteo.global.C4} gives $\deg(G)\geq 2+2t$, so
	$t\leq\lfloor(\deg(G)-2)/2\rfloor$. For $\deg(G)\leq 3$ this yields
	$t=0$, that is, $V(C)$ is a module; for $\deg(G)\leq 5$ it yields
	$t\leq 1$. Finally, if $\deg(G)\leq 5$ and a vertex $v$ had
	$N_G(v)\cap V(C)$ equal to a pair of opposite vertices of $C$, then
	its contribution would be $4$ by \Cref{conteo.local.C4}, whence
	$\deg(G)\geq 2+4=6$, a contradiction.
\end{proof}

\begin{lemma}
	\label{C4.modulo.deg>=6}
	Let $G$ be a prime graph and let $C\prec G$ with $C\approx C_4$. If
	$V(C)$ is a module of $G$ and $G\neq C$, then $\deg(G)\geq 6$.
\end{lemma}

\begin{proof}
	Put $M=V(C)$. Since $G\neq C$, we have $V(G)-M\neq\varnothing$, and
	since $A_4(G)$ is connected by \Cref{indecomp.characterization}, some
	$H\in \mathcal{Q}_G$ meets both $M$ and $V(G)-M$. Put $W=V(H)\cap M$, so
	that $1\leq|W|\leq 3$. If $|W|=3$, the vertex of $V(H)-W$ is adjacent
	to all or to none of the vertices of $W$, which is impossible because the vertices of
	$P_4$, $C_4$ and $2K_2$ have degree 1 or 2. If $|W|=2$, say $W=\{w_1,w_2\}$, then since $M$ is a module
	every vertex of $V(H)-W$ is adjacent to both or to neither of $w_1,w_2$;
	as no two vertices of $P_4$ have this property, $\deg(H)=2$. By \Cref{modulo.reemplazo},
	each $A\subseteq M$ with $G[A]\approx G[W]$ yields a member
	$G[(V(H)-W)\cup A]$ of $\mathcal{Q}_G$ isomorphic to $H$ and distinct from
	$C$; since $C_4$ has four pairs of vertices inducing an edge and two
	inducing a non-edge, there are at least two such $A$, whence
	$\deg(G)\geq 2+2\cdot 2=6$ by \eqref{deg.como.suma}. If $|W|=1$,
	\Cref{modulo.reemplazo} applied to the four singletons
	$A\subseteq M$ yields four members of $\mathcal{Q}_G$ distinct from $C$, so
	$\deg(G)\geq 2+4=6$.
\end{proof}


\section{Active graphs of degree 1 and 2}\label{sec_deg<3}

An immediate consequence of  \Cref{|G^*|} is that $P_4$ is the only active graph of degree 1. Then, by \Cref{hechos.previos}, $P_4^2$ is the only decomposable active graph of degree 2.
Now, consider a prime graph $G$ such that $\deg(G)=|\mathcal{Q}_G(P_4)|=2$. By  \Cref{5>degG=|{P_4}|.implica.G.split}, $G$ is split. 

Then $\Phi(G)$ is connected, by  \Cref{S.primo.iff.Phi(S).conexo}. This means that $\Phi(G)$ is isomorphic to $K_2$ or $P_3$, ignoring multiplicities. If $(K,I)$ is the bipartition for $G$, we have $(|K|,|I|)=(\omega,\alpha)$, where $\omega=\omega(G)$ and $\alpha=\alpha(G)$. Assume $\alpha\leq\omega$. As $5 \leq \alpha+\omega=|G|$, it follows that $\omega \neq 2$. Thus, $(\omega,\alpha)\in \{(3,2), (3,3)\}$ by \Cref{|I|_|K|<=2deg(S)}.
If $\alpha=2$, then $G\approx D_{2,1}$ by \Cref{S_activo_|I|=2} (see \Cref{grafos.activos.deg<3}).
If $\alpha=3$, then $\Phi=abc$, and \Cref{S_activo_|I|=3} gives $\kappa_b(\pi_b+\kappa_a)=1=\pi_b(\pi_c+\kappa_b)$. Hence $\pi_b=\kappa_b=1$ and $\kappa_a=\pi_c=0$, so $\omega=2$, which is absurd.



So, if $G$ is a prime split graph of degree 2, then $G\approx D_{2,1}$ or $\overline{D_{2,1}}$. If $G$ is a prime non-split graph of degree 2, it is clear that $|\mathcal{Q}_G|=1$. Moreover, $|G|=4$, since otherwise $|\mathcal{Q}_G|\geq 2$. Then, $G\approx C_4$ or $2K_2$.

\begin{theorem}
	\label{clasificacion.activo_deg=2}
	Let $G$ be an active graph.
	\begin{enumerate}
		\item If $\deg(G)=1$, then $G \approx P_4$.
		\item If $\deg(G)=2$ and $G$ is decomposable, then $G \approx P_4^2$.
		\item If $\deg(G)=2$ and $G$ is indecomposable, then $G$ is isomorphic to one of these graphs: $D_{2,1}$, $\overline{D_{2,1}}$, $C_4$, $2K_2$.
	\end{enumerate}
	All graphs mentioned here are represented in  \Cref{grafos.activos.deg<3}.
\end{theorem}

\begin{proof}
	It follows from the previous discussion.
\end{proof}

\begin{figure}[h!]
	\centering
	\begin{tikzpicture}[
		scale=1,
		every node/.style={circle, draw, inner sep=2pt},
		lbl/.style={draw=none}
		]
		
		\begin{scope}[shift={(0,2)}]
			\node (t) at (0,1.2) {};
			\node (l) at (-0.6,0.4) {};
			\node (r) at (0.6,0.4) {};
			\node (b) at (0,-0.4) {};
			\node (e) at (0.6,-0.4) {};
			
			\draw (t)--(l)--(b)--(r)--(t);
			\draw (l)--(r);
			\draw (b)--(e);
			
			\node[lbl] at (1.2,-0.4) {$D_{2,1}$};
		\end{scope}
		
		\begin{scope}[shift={(2.4,2)}]
			\node (t) at (0,1.2) {};
			\node (l) at (-0.6,0.4) {};
			\node (r) at (0.6,0.4) {};
			\node (b) at (0,-0.4) {};
			\node (e) at (0.6,-0.4) {};
			
			\draw (t)--(r);
			\draw (l)--(r);
			\draw (r)--(b);
			\draw (b)--(e);
			
			\node[lbl] at (1.2,0.4) {$\overline{D_{2,1}}$};
		\end{scope}
		
		\begin{scope}[shift={(4.8,2)}]
			
			\node (t1) at (0,1) {};
			\node (t2) at (1,1) {};
			\node (t3) at (2,1) {};
			\node (t4) at (3,1) {};
			
			\node (b1) at (0,0) {};
			\node (b2) at (1,0) {};
			\node (b3) at (2,0) {};
			\node (b4) at (3,0) {};
			
			\draw (t1)--(t2)--(t3)--(t4);
			\draw (b1)--(b2)--(b3)--(b4);
			
			\draw (t2)--(b2);
			\draw (t3)--(b3);
			
			\draw (b1)--(t2);
			\draw (b1)--(t3);
			\draw (b2)--(t3);
			\draw (b4)--(t2);
			\draw (b3)--(t2);
			\draw (b4)--(t3);
			
			\node[lbl] at (1.5,-0.7) {$P_4^2$};
		\end{scope}
		
		\begin{scope}[shift={(0,-1)}]
			\node (a) at (0,1) {};
			\node (b) at (1,1) {};
			\node (c) at (1,0) {};
			\node (d) at (0,0) {};
			
			\draw (a)--(b)--(c)--(d)--(a);
			
			\node[lbl] at (0.5,1.4) {$C_4$};
		\end{scope}
		
		\begin{scope}[shift={(2.4,-1)}]
			\node (a) at (0,1) {};
			\node (b) at (1,1) {};
			\node (c) at (0,0) {};
			\node (d) at (1,0) {};
			
			\draw (a)--(b);
			\draw (c)--(d);
			
			\node[lbl] at (0.5,1.4) {$2K_2$};
		\end{scope}
		
		\begin{scope}[shift={(4.8,-1)}]
			\node (a) at (0,0) {};
			\node (b) at (1,0) {};
			\node (c) at (2,0) {};
			\node (d) at (3,0) {};
			
			\draw (a)--(b)--(c)--(d);
			
			\node[lbl] at (1.5,0.6) {$P_4$};
		\end{scope}
		
	\end{tikzpicture}
	\caption{The 6 active graphs of degree 1 or 2.}
	\label{grafos.activos.deg<3}
\end{figure}
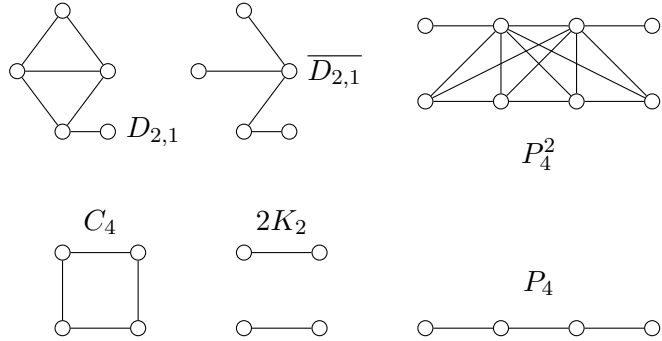


\section{Active graphs of degree 3}\label{sec_deg=3}

Consider an active graph $G$ such that $\deg(G)=|\mathcal{Q}_G(P_4)|=3$. Thanks to  \Cref{5>degG=|{P_4}|.implica.G.split}, we know that $G$ is split. Suppose $G$ is decomposable. Since $G$ is active, it follows from \Cref{hechos.previos} that each factor of $G$ is an active split graph of degree 1 or 2. Consequently,
\[ G \in \{P_4 \circ D_{2,1}, D_{2,1} \circ P_4, P_4 \circ \overline{D_{2,1}}, \overline{D_{2,1}} \circ P_4, P_4^3\}, \]
by  \Cref{clasificacion.activo_deg=2}.

If instead $(G,K,I)$ is indecomposable, assume $\alpha=|I|\leq|K|=\omega$. Clearly, $\omega \geq 3$, since otherwise we would have $G\approx P_4$. Thus, $2 \leq \alpha \leq 4$ and $\omega \in \{3,4\}$, by \Cref{|I|_|K|<=2deg(S)}. Applying \Cref{S.primo.iff.Phi(S).conexo}, we have
\[ \Phi(G) \in \{K_2, K_3, P_3, K_{3,1}, P_4\}, \]
ignoring multiplicities if any (see \Cref{Phi.posibles.deg=3}).
\begin{figure}[h]
	\centering
	\begin{tikzpicture}[scale=1, every node/.style={draw, circle, fill=white, minimum size=0.25cm, inner sep=1pt}]
		
		\begin{scope}[shift={(0,0)}]
			\node (a1) at (0,0) {};
			\node (a2) at (0,1) {};
			\draw (a1) -- (a2);
			\draw[bend left=25] (a1) to (a2);
			\draw[bend right=25] (a1) to (a2);
		\end{scope}
		
		\begin{scope}[shift={(1.7,0)}]
			\node (b1) at (0,0) {};
			\node (b2) at (1,0) {};
			\node (b3) at (0,1) {};
			\draw (b1) -- (b2);
			\draw (b1) -- (b3);
			\draw[bend right=25] (b1) to (b3);
		\end{scope}
		
		\begin{scope}[shift={(4.3,0)}]
			\node (c1) at (0,0) {};
			\node (c2) at (1,0) {};
			\node (c3) at (0,1) {};
			\draw (c1) -- (c2);
			\draw (c1) -- (c3);
			\draw (c2) -- (c3);
		\end{scope}
		
		\begin{scope}[shift={(6.9,0)}]
			\node (d1) at (0,0) {};
			\node (d2) at (1,0) {};
			\node (d3) at (0,1) {};
			\node (d4) at (1,1) {};
			\draw (d1) -- (d2);
			\draw (d1) -- (d3);
			\draw (d2) -- (d4);
		\end{scope}
		
		\begin{scope}[shift={(9.5,0)}]
			\node (e0) at (0.5,0.5) {};
			\node (e1) at (0,1) {};
			\node (e2) at (1,1) {};
			\node (e3) at (0.5,0) {};
			\draw (e0) -- (e1);
			\draw (e0) -- (e2);
			\draw (e0) -- (e3);
		\end{scope}
		
	\end{tikzpicture}
	\caption{The 5 connected multigraphs of size 3.}
	\label{Phi.posibles.deg=3}
\end{figure}
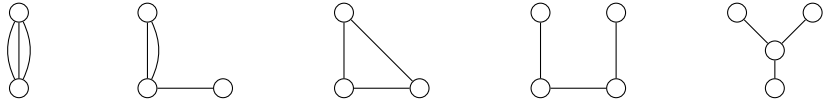
If $\alpha=2$, then $G\approx D_{3,1}$ by \Cref{S_activo_|I|=2} (see \Cref{grafos.primos.deg=3}).
If $\Phi=\Phi(G) \in \{K_3, K_{3,1}, P_4\}$, then $G$ is $\delta$-homogeneous, by  \Cref{hechos.previos}.  The same proposition therefore prohibits $\Phi=P_4$. Suppose $\Phi=K_{3,1}$. Thanks to  \Cref{hechos.previos}, we know that $N_1=N_2=N_3$. On the other hand, $0=|\bigcap_{i=1}^4 N_i|=\eta_{14}$, because $G$ has no universal vertices. This means that $\delta=1$, by  \Cref{prop.basicas.sigma_uv}. But then $\omega=2$, contradicting $\alpha \leq \omega$. Hence, $\alpha\neq 4$.

If $\alpha=3$, then $\Phi=abca$, or $\Phi=abc$ with $\{\sigma_{ab},\sigma_{bc}\}=\{1,2\}$. If $\Phi=abca$ then the six factors in \Cref{S_activo_|I|=3}(2) are equal to $1$, so $\pi_a=\pi_b=\pi_c$, $\kappa_a=\kappa_b=\kappa_c$ and $\pi_a+\kappa_a=1$. This yields two graphs, complementary to each other; we call $T_{111}$ the one with $\pi_a=1$ (the subscript records the degrees of the $I$-vertices). If $\Phi=abc$, then, exchanging $a$ and $c$ if necessary, \Cref{S_activo_|I|=3}(4) gives $\sigma_{ab}=\kappa_b(\pi_b+\kappa_a)$ and $\sigma_{bc}=\pi_b(\pi_c+\kappa_b)$, so $\pi_b=\kappa_b=1$. If $\sigma_{ab}=2$, then $\kappa_a=1$ and $\pi_c=0$, and we call the resulting graph $T_{221}$; if $\sigma_{ab}=1$, then $\kappa_a=0$ and $\pi_c=1$, which yields $\overline{T_{221}}$. In all cases $\omega=3$. These graphs are shown in \Cref{grafos.primos.deg=3}.
\begin{figure}[ht]
	\centering
	\begin{tikzpicture}[scale=0.9, 
		every node/.style={draw, circle, fill=white, inner sep=0pt, minimum size=6.5pt},
		]
		
		\begin{scope}[shift={(0,3.5)}]
			\node (t1) at (0, 1) {};
			\node (t2) at (-0.55, 0) {};
			\node (t3) at (0.55, 0) {};
			\node (p1) at (0, 2) {};
			\node (p2) at (-1.25, -0.7) {};
			\node (p3) at (1.25, -0.7) {};
			\draw (t1) -- (t2) -- (t3) -- (t1);
			\draw (t1) -- (p1);
			\draw (t2) -- (p2);
			\draw (t3) -- (p3);
			\node[draw=none] at (-1, 1) {$T_{111}$};
		\end{scope}
		
		\begin{scope}[shift={(2.2,3)}]
			\node (a)  at (0, 2.2) {};
			\node (k1) at (1.0, 1.3) {};
			\node (k2) at (2.2, 1.3) {};
			\node (k3) at (1.0, 0) {};
			\node (k4) at (2.2, 0) {};
			\node (pp) at (3.2, 0) {};
			\draw (k1) -- (k2) -- (k4) -- (k3) -- (k1);
			\draw (k1) -- (k4);
			\draw (k2) -- (k3);
			\draw (a) -- (k1);
			\draw (a) -- (k2);
			\draw (a) -- (k3);
			\draw (k4) -- (pp);
			\node[draw=none] at (2, 2) {$D_{3,1}$};
		\end{scope}
		
		\begin{scope}[shift={(6.2,3.5)}]
			\node (u1) at (0, 1.3) {};
			\node (u2) at (1, 1.3) {};
			\node (u3) at (2, 1.3) {};
			\node (u4) at (0.5, 0.3) {};
			\node (u5) at (1.5, 0.3) {};
			\node (u6) at (1, -0.7) {};
			\draw (u1) -- (u2) -- (u3);
			\draw (u1) -- (u4);
			\draw (u3) -- (u5);
			\draw (u2) -- (u4);
			\draw (u2) -- (u5);
			\draw (u4) -- (u5);
			\draw (u4) -- (u6);
			\node[draw=none] at (2, -0.7) {$T_{221}$};
		\end{scope}
		
		\begin{scope}[shift={(0,0.5)}]
			\node (b1) at (-1.2, 1) {};
			\node (b2) at (0, 1) {};
			\node (b3) at (1.2, 1) {};
			\node (b4) at (-0.55, 0) {};
			\node (b5) at (0.55, 0) {};
			\node (b6) at (0, -1.1) {};
			\draw (b1) -- (b2) -- (b3);
			\draw (b1) -- (b4);
			\draw (b3) -- (b5);
			\draw (b2) -- (b4);
			\draw (b2) -- (b5);
			\draw (b4) -- (b5);
			\draw (b4) -- (b6);
			\draw (b5) -- (b6);
			\node[draw=none] at (-1, -1) {$\overline{T_{111}}$};
		\end{scope}
		
		\begin{scope}[shift={(2.6,0)}]
			\node (c1) at (0, 1.5) {};
			\node (c2) at (0.9, 0.6) {};
			\node (c3) at (2.0, 0.6) {};
			\node (c4) at (0.9, -0.6) {};
			\node (c5) at (2.0, -0.6) {};
			\node (c6) at (3.1, -0.6) {};
			\draw (c1) -- (c2);
			\draw (c2) -- (c3);
			\draw (c2) -- (c4);
			\draw (c2) -- (c5);
			\draw (c5) -- (c6);
			\node[draw=none] at (-0.1, 0) {$\overline{D_{3,1}}$};
		\end{scope}
		
		\begin{scope}[shift={(6.2,0)}]
			\node (d1) at (0, 1.3) {};
			\node (d2) at (1, 1.3) {};
			\node (d3) at (2, 1.3) {};
			\node (d4) at (0.5, 0.3) {};
			\node (d5) at (1.5, 0.3) {};
			\node (d6) at (1, -0.7) {};
			\draw (d2) -- (d3);
			\draw (d1) -- (d4);
			\draw (d2) -- (d4);
			\draw (d2) -- (d5);
			\draw (d4) -- (d5);
			\draw (d4) -- (d6);
			\draw (d5) -- (d6);
			\node[draw=none] at (2, -0.7) {$\overline{T_{221}}$};
		\end{scope}
		
	\end{tikzpicture}
	\caption{The 6 prime graphs of degree $3$.}
	\label{grafos.primos.deg=3}
\end{figure}
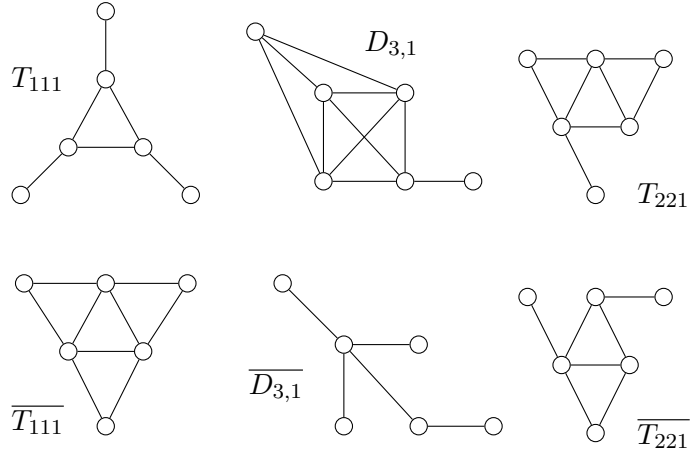

So far, we have classified active graphs $G$ of degree 3 such that $\mathcal{Q}_G=\mathcal{Q}_G(P_4)$. To complete the classification of all active graphs of degree 3, it remains to analyze the following case:
\begin{equation}
	\label{eq38}
	|\mathcal{Q}_G(C_4)| + |\mathcal{Q}_G(2K_2)| = 1 = |\mathcal{Q}_G(P_4)|.
\end{equation}
Recall that, by definition of $\circ$, if $G=X \circ Y$, then $X$ is split. So, if $G$ is decomposable and satisfies \eqref{eq38}, then $G \in \{P_4 \circ C_4, P_4 \circ 2K_2\}$, by \Cref{clasificacion.activo_deg=2} and \Cref{hechos.previos}. In the following lemma, we prove that there are no prime graphs satisfying \eqref{eq38}.

%
%

\begin{lemma}
	\label{no.existe.deg3.primo}
	There exists no prime graph satisfying \eqref{eq38}.
\end{lemma}

\begin{proof}
	Assume by contradiction that $G$ is prime and satisfies \eqref{eq38}. A
	$4$-subset of $V(G)$ induces $P_4$, $C_4$ or $2K_2$ in $G$ if and only if
	it induces $P_4$, $2K_2$ or $C_4$ in $\overline{G}$, respectively, and $G$
	is prime if and only if $\overline{G}$ is; hence we may assume
	$\mathcal{Q}_G=\{P,C\}$ with $P\approx P_4$ and $C\approx C_4$. By
	\eqref{deg.como.suma}, $\deg(G)=3$, so $V(C)$ is a module
	of $G$ by \Cref{C4.distinguidores}. Moreover $G\neq C$, because
	$\deg(C)=2\neq\deg(G)$. Therefore $\deg(G)\geq 6$ by
	\Cref{C4.modulo.deg>=6}, a contradiction.
\end{proof}

\begin{theorem}
	\label{clasificacion.activo_deg=3}
	Let $G$ be an active graph of degree 3.
	\begin{enumerate}
		\item If $G$ is decomposable, then $G \in$ \[ \{P_4 \circ D_{2,1}, D_{2,1} \circ P_4, P_4 \circ \overline{D_{2,1}}, \overline{D_{2,1}} \circ P_4, P_4^3, P_4 \circ C_4, P_4 \circ 2K_2\}, \]
		\item If $G$ is indecomposable, then
		\[ G \in \{T_{111}, \overline{T_{111}}, D_{3,1}, \overline{D_{3,1}}, T_{221}, \overline{T_{221}}\}, \]
		all of which are split (see  \Cref{grafos.primos.deg=3}).
	\end{enumerate}
\end{theorem}

\begin{proof}
	It follows from the previous discussion and \Cref{no.existe.deg3.primo}.
\end{proof}


\section{Prime split graphs of degree 4}\label{sec_deg=4}

We now classify the prime split graphs of degree 4. Consider a prime graph $G$ such that $\deg(G) = |\mathcal{Q}_G(P_4)| = 4$. Thanks to  \Cref{5>degG=|{P_4}|.implica.G.split}, we know that $G$ is split. Assume $G$ has bipartition $(K,I)$ and $\alpha = |I| \leq |K| = \omega$. Clearly, $\omega \geq 3$, since otherwise we would have $G\approx P_4$. Thus, $2 \leq \alpha \leq 5$ and $3 \leq \omega \leq 5$, by \Cref{|I|_|K|<=2deg(S)}. Applying \Cref{S.primo.iff.Phi(S).conexo}, we have that $\Phi = \Phi(G)$ is isomorphic to one of the 12 multigraphs in  \Cref{Phi.posibles.deg=4}.
\begin{figure}[h]
	\centering
	\begin{tikzpicture}[scale=1, every node/.style={draw, circle, fill=white, minimum size=0.25cm, inner sep=1pt}]
		
		\node (a1) at (0,0) {};
		\node (a2) [right=of a1] {};
		\draw (a1) -- (a2);
		\draw[bend left=15] (a1) to (a2);
		\draw[bend left=35] (a1) to (a2);
		\draw[bend right=15] (a1) to (a2);
		
		\node (b1) [right=2cm of a2] {};
		\node (b2) [right=of b1] {};
		\node (b3) [above=of b1] {};
		\draw (b1) -- (b2);
		\draw (b1) -- (b3);
		\draw[bend left=20] (b1) to (b2);
		\draw[bend right=20] (b1) to (b3);
		
		\node (c1) [right=2cm of b2] {};
		\node (c2) [right=of c1] {};
		\node (c3) [above=of c1] {};
		\draw (c1) -- (c2);
		\draw (c1) -- (c3);
		\draw (c2) -- (c3);
		\draw[bend left=20] (c2) to (c3);
		
		\node (d1) at (0,-2) {};
		\node (d2) [right=of d1] {};
		\node (d3) [above=of d1] {};
		\node (d4) [above=of d2] {};
		\draw (d1) -- (d2);
		\draw (d1) -- (d3);
		\draw (d2) -- (d4);
		\draw (d3) -- (d4);
		
		\node (e1) [right=2cm of d2] {};
		\node (e2) [right=of e1] {};
		\node (e3) [above=of e1] {};
		\node (e4) [above=of e2] {};
		\draw (e4) -- (e3);
		\draw (e1) -- (e3);
		\draw[bend left=20] (e1) to (e3);
		\draw (e2) -- (e3);
		
		\node (f1) [right=2cm of e2] {};
		\node (f2) [right=of f1] {};
		\node (f3) [above=of f1] {};
		\draw (f1) -- (f3);
		\draw (f2) -- (f3);
		\draw[bend left=20] (f3) to (f1);
		\draw[bend right=20] (f3) to (f1);
		
		\node (g1) at (0,-4) {};
		\node (g2) [right=of g1] {};
		\node (g3) [above=of g1] {};
		\node (g4) [above=of g2] {};
		\draw (g1) -- (g2);
		\draw (g1) -- (g3);
		\draw (g3) -- (g4);
		\draw[bend left=20] (g3) to (g1);
		
		\node (h1) [right=2cm of g2] {};
		\node (h2) [right=of h1] {};
		\node (h3) [above=of h1] {};
		\node (h4) [above=of h2] {};
		\draw (h1) -- (h2);
		\draw (h1) -- (h3);
		\draw (h2) -- (h4);
		\draw (h2) -- (h3);
		
		\node (i1) [right=2cm of h2] {};
		\node (i2) [right=of i1] {};
		\node (i3) [above=of i1] {};
		\node (i4) [above=of i2] {};
		\draw (i1) -- (i2);
		\draw (i1) -- (i3);
		\draw (i2) -- (i4);
		\draw[bend left=20] (i1) to (i3);
		
		\node (j1) at (0,-6) {};
		\node (j2) [right=of j1] {};
		\node (j3) [above=of j1] {};
		\node (j4) [above=of j2] {};
		\node (j5) [below=of j1] {};
		\draw (j1) -- (j5);
		\draw (j1) -- (j3);
		\draw (j2) -- (j4);
		\draw (j3) -- (j4);
		
		\node (k1) [right=2cm of j2] {};
		\node (k2) [right=of k1] {};
		\node (k3) [above=of k1] {};
		\node (k4) [above=of k2] {};
		\node (k5) [below=of k1] {};
		\draw (k1) -- (k2);
		\draw (k1) -- (k3);
		\draw (k1) -- (k4);
		\draw (k1) -- (k5);
		
		\node (l1) [right=2cm of k2] {};
		\node (l2) [right=of l1] {};
		\node (l3) [above=of l1] {};
		\node (l4) [above=of l2] {};
		\node (l5) [below=of l1] {};
		\draw (l5) -- (l2);
		\draw (l1) -- (l3);
		\draw (l1) -- (l4);
		\draw (l1) -- (l5);
	\end{tikzpicture}
	\caption{The 12 connected multigraphs of size 4.}
	\label{Phi.posibles.deg=4}
\end{figure}
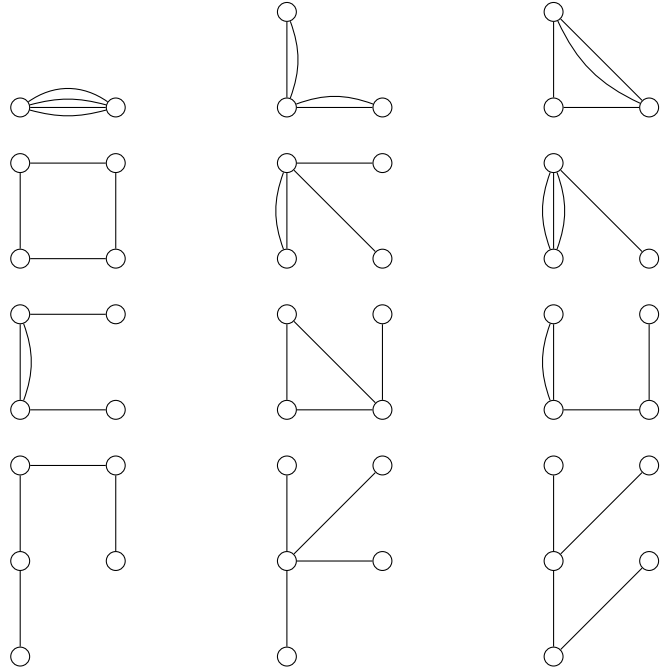
If $\alpha=2$, then $G\approx D_{4,1}$ or $D_{2,2}$ by \Cref{S_activo_|I|=2} (see \Cref{split.primos.deg=4.fig.1}). Thanks to \Cref{hechos.previos}, we can immediately rule out that $\Phi$ is the $K_3$ with a pendant vertex, $P_5$, or $\overline{D_{2,1}}$. 
Applying Theorem 4.1 of \cite{schvollner2026induced}, we can also exclude that $\Phi = v_1v_2v_3v_4$ with $\sigma_{23} = 1$. If $\Phi = abca$ with $\sigma_{ac} = 2$, then by  \Cref{prop.basicas.sigma_uv} we would have $d_a = d_b = d_c = d_a \pm 1$, which is absurd.

If $\Phi \approx K_{4,1}$, let $v_5$ be the vertex of degree 4 in $K_{4,1}$ and $v_i$ its leaves, $i \in [4]$. Applying \Cref{hechos.previos}, we have that the $N_i=N_j$ for all $i,j\in [4]$, so $K = N_1 \cup N_5$. On the other hand, since $G$ has no universal vertices, we have $\varnothing = \bigcap_{j=1}^5 N_j = N_1 \cap N_5$, and consequently, $d_1 = 1 = d_5$, by  \Cref{prop.basicas.sigma_uv}. Then, $5 = \alpha \leq \omega = 2$, an absurdity.

If $\Phi = v_1v_2v_3v_4v_1$, then $N_1 = N_3$ and $N_2 = N_4$, by \Cref{hechos.previos}, so $K = N_1 \cup N_2$. Since $G$ has no universal vertices, we have $\varnothing = \bigcap_{i=1}^4 N_i = N_1 \cap N_2$, and consequently, $d_1 = 1 = d_2$, by  \Cref{prop.basicas.sigma_uv}. Then, $4 = \alpha \leq \omega = d_1 + d_2 = 2$, which is impossible.

If $\Phi = v_1v_2v_3v_4$ with $\sigma_{23} = 2$, then we can assume without loss of generality that $d_2 \leq d_3$. Then, $d_1 = d_2$, $d_3 = d_1 + 1 = d_4$, $N_2 \subseteq N_4$, and $N_1 \subseteq N_3, N_4$, by  \Cref{prop.basicas.sigma_uv}. Thus, $K = N_3 \cup N_4$. Since $G$ has no universal vertices, we have $\varnothing = \bigcap_{i=1}^4 N_i = N_1 \cap N_2$, and consequently, $d_1 = 1$, by  \Cref{prop.basicas.sigma_uv}. Thus, $d_3 = 2$ and $\eta_{34} = 1$, again by  \Cref{prop.basicas.sigma_uv}. Finally, we obtain $4 = \alpha \leq \omega = d_3 + d_4 - \eta_{34} = 3$, an absurdity.

If $E(\Phi) = \{2v_1v_2, v_2v_3, v_2v_4\}$ (a $K_{3,1}$ with a double-edge), then $d_3 = d_2 = d_4$, $|d_2 - d_1| = 1$, and $N_3 = N_4$, by  \Cref{prop.basicas.sigma_uv}. If $d_2 = d_1 + 1$, then $N_1 \subseteq N_3$, so $K = N_2 \cup N_3$. Since $G$ has no universal vertices, we have $\varnothing = \bigcap_{i=1}^4 N_i = N_1 \cap N_2$, and consequently, $d_1 = 1$, by  \Cref{prop.basicas.sigma_uv}. Thus, $\eta_{23} = 1$, again by  \Cref{prop.basicas.sigma_uv}. Finally, we obtain $4 = \alpha \leq \omega = d_2 + d_3 - \eta_{23} = 3$, which is impossible. The case $d_1 = d_2 + 1$ is reduced to an absurdity with similar arguments.

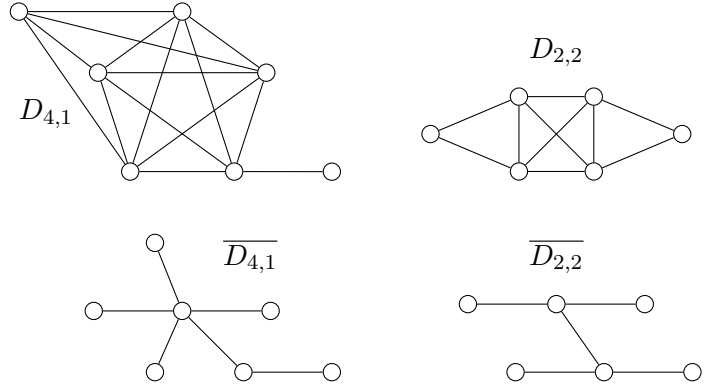
\begin{figure}[ht]
	\centering
	\begin{tikzpicture}[scale=0.9, 
		every node/.style={draw, circle, fill=white, inner sep=0pt, minimum size=6.5pt},
		]
		
		\begin{scope}[shift={(0,3.5)}]
			\node (t)  at (90:1.3)  {}; 
			\node (ur) at (18:1.3)  {}; 
			\node (lr) at (-54:1.3) {}; 
			\node (ll) at (234:1.3) {}; 
			\node (ul) at (162:1.3) {}; 
			\node (v1) at (-2.4, 1.3)  {}; 
			\node (v2) at ( 2.2, -1.053) {}; 
			\draw (t)--(ur); \draw (ur)--(lr); \draw (lr)--(ll); \draw (ll)--(ul); \draw (ul)--(t);
			\draw (t)--(lr); \draw (t)--(ll); \draw (ur)--(ll); \draw (ur)--(ul); \draw (lr)--(ul);
			\draw (v1)--(t);
			\draw (v1)--(ul);
			\draw (v1)--(ur);
			\draw (v1)--(ll);
			\draw (v2)--(lr);
			\node[draw=none] at (-2.0, -0.2) {$D_{4,1}$};
		\end{scope}
		
		\begin{scope}[shift={(5.5,3)}]
			\node (ul) at (-0.55, 0.55) {};
			\node (ur) at ( 0.55, 0.55) {};
			\node (lr) at ( 0.55,-0.55) {};
			\node (ll) at (-0.55,-0.55) {};
			\node (w1) at (-1.85, 0)    {};
			\node (w2) at ( 1.85, 0)    {};
			\draw (ul)--(ur); \draw (ur)--(lr); \draw (lr)--(ll); \draw (ll)--(ul);
			\draw (ul)--(lr); \draw (ur)--(ll);
			\draw (w1)--(ul); \draw (w1)--(ll);
			\draw (w2)--(ur); \draw (w2)--(lr);
			\node[draw=none] at (0, 1.2) {$D_{2,2}$};
		\end{scope}
		
		\begin{scope}[shift={(0,0.4)}]
			\node (A)  at (0, 0)      {}; 
			\node (p1) at (-0.4, 1)    {}; 
			\node (p2) at (-1.3, 0)   {}; 
			\node (p3) at ( 1.3, 0)   {}; 
			\node (p4) at (-0.4,-0.9){}; 
			\node (B)  at ( 0.9,-0.9) {}; 
			\node (p5) at ( 2.2,-0.9) {}; 
			\draw (A)--(p1);
			\draw (A)--(p2);
			\draw (A)--(p3);
			\draw (A)--(p4);
			\draw (A)--(B);
			\draw (B)--(p5);
			\node[draw=none] at (1, 0.8) {$\overline{D_{4,1}}$};
		\end{scope}
		
		\begin{scope}[shift={(5.5,0)}]
			\node (A)  at (0, 0.5)    {};
			\node (q1) at (-1.3, 0.5) {};
			\node (q2) at ( 1.3, 0.5) {};
			\node (B)  at ( 0.7,-0.5) {};
			\node (q3) at (-0.6,-0.5) {};
			\node (q4) at ( 2.0,-0.5) {};
			\draw (A)--(q1);
			\draw (A)--(q2);
			\draw (A)--(B);
			\draw (B)--(q3);
			\draw (B)--(q4);
			\node[draw=none] at (0, 1.2) {$\overline{D_{2,2}}$};
		\end{scope}
		
	\end{tikzpicture}
	\caption{The prime split graphs $D_{4,1}, D_{2,2}$ and their complements.}
	\label{split.primos.deg=4.fig.1}
\end{figure}

If $\Phi=abc$ ignoring multiplicities, then, exchanging $a$ and $c$ if necessary, \Cref{S_activo_|I|=3} gives $\sigma_{ab}=\kappa_b(\pi_b+\kappa_a)$ and $\sigma_{bc}=\pi_b(\pi_c+\kappa_b)$, with $\sigma_{ab}+\sigma_{bc}=4$ and $\pi_b,\kappa_b\geq 1$. A direct computation yields exactly the five graphs of the following table, where only the parameters of \Cref{S_activo_|I|=3} not forced to vanish are listed; as in \Cref{sec_deg=3}, the subscript of each name records the degrees of the $I$-vertices. In all cases $\omega\geq 3=\alpha$.
\begin{center}
	\begin{tabular}{cccc}
		\hline
		$(\sigma_{ab},\sigma_{bc})$ & $(\pi_b,\pi_c,\kappa_a,\kappa_b)$ & $\omega$ & $G$ \\
		\hline
		$(1,3)$ & $(1,2,0,1)$ & $4$ & $R_{311}$ \\
		$(3,1)$ & $(1,0,2,1)$ & $4$ & $R_{331}$ \\
		$(2,2)$ & $(1,1,1,1)$ & $4$ & $R_{321}$ \\
		$(2,2)$ & $(2,0,0,1)$ & $3$ & $R_{211}$ \\
		$(2,2)$ & $(1,0,0,2)$ & $3$ & $\overline{R_{211}}$ \\
		\hline
	\end{tabular}
\end{center}
These graphs, together with $\overline{R_{311}}$, $\overline{R_{331}}$ and $\overline{R_{321}}$, are shown in \Cref{split.primos.deg=4.fig.2,split.primos.deg=4.fig.3}.
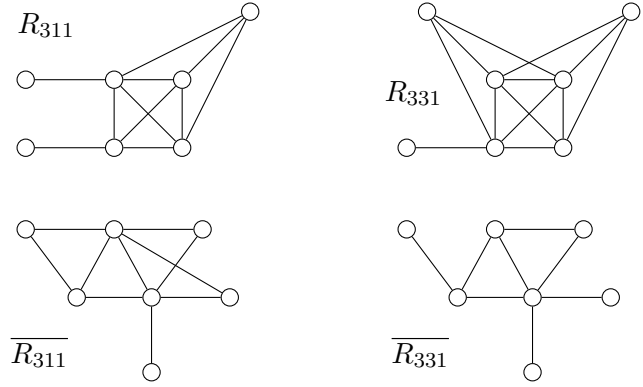
\begin{figure}[h]
	\centering
	\begin{tikzpicture}[scale=0.9, 
		every node/.style={draw, circle, fill=white, inner sep=0pt, minimum size=6.5pt},
		]
		
		\begin{scope}[shift={(1,3.2)}]
			\node (ul) at (0, 1)  {};
			\node (ur) at (1, 1)  {};
			\node (ll) at (0, 0)  {};
			\node (lr) at (1, 0)  {};
			\node (w1) at (-1.3, 1)   {}; 
			\node (w2) at (-1.3, 0)   {}; 
			\node (w3) at ( 2, 2) {}; 
			\draw (ul)--(ur); \draw (ur)--(lr); \draw (lr)--(ll); \draw (ll)--(ul);
			\draw (ul)--(lr); \draw (ur)--(ll);
			\draw (w1)--(ul);
			\draw (w2)--(ll);
			\draw (w3)--(ul); \draw (w3)--(ur); \draw (w3)--(lr);
			\node[draw=none] at (-1.0, 1.8) {$R_{311}$};
		\end{scope}
		
		\begin{scope}[shift={(6.6,3.2)}]
			\node (ul) at (0, 1)  {};
			\node (ur) at (1, 1)  {};
			\node (ll) at (0, 0)  {};
			\node (lr) at (1, 0)  {};
			\node (w1) at (-1, 2) {}; 
			\node (w2) at (2, 2) {}; 
			\node (w3) at (-1.3, 0)   {}; 
			\draw (ul)--(ur); \draw (ur)--(lr); \draw (lr)--(ll); \draw (ll)--(ul);
			\draw (ul)--(lr); \draw (ur)--(ll);
			\draw (w1)--(ul); \draw (w1)--(ur);
			\draw (w2)--(ul); \draw (w2)--(ur);
			\draw (w3)--(ll); \draw (ll)--(w1); 
			\draw (lr)--(w2);
			\node[draw=none] at (-1.2, 0.8) {$R_{331}$};
		\end{scope}
		
		\begin{scope}[shift={(1,1)}]
			\node (A) at (0, 1)      {};
			\node (B) at (-0.55, 0)  {};
			\node (C) at ( 0.55, 0)  {};
			\node (p1) at (-1.3, 1)  {};
			\node (p2) at ( 1.3, 1)  {};
			\node (p3) at ( 1.7, 0)  {};
			\node (p4) at ( 0.55,-1.1) {};
			\draw (A)--(B); \draw (B)--(C); \draw (A)--(C);
			\draw (A)--(p1);
			\draw (A)--(p2);
			\draw (A)--(p3);   
			\draw (B)--(p1);
			\draw (C)--(p2);   
			\draw (C)--(p3);
			\draw (C)--(p4);
			\node[draw=none] at (-1.1,-0.8) {$\overline{R_{311}}$};
		\end{scope}
		
		\begin{scope}[shift={(6.6,1)}]
			\node (A) at (0, 1)      {};
			\node (B) at (-0.55, 0)  {};
			\node (C) at ( 0.55, 0)  {};
			\node (q1) at (-1.3, 1)  {};
			\node (q2) at ( 1.3, 1)  {};
			\node (q3) at ( 1.7, 0)  {};
			\node (q4) at ( 0.55,-1.1) {};
			\draw (A)--(B); \draw (B)--(C); \draw (A)--(C);
			\draw (A)--(q2);
			\draw (B)--(q1);
			\draw (C)--(q3);
			\draw (C)--(q4);
			\draw (C)--(q2);
			\node[draw=none] at (-1.1,-0.8) {$\overline{R_{331}}$};
		\end{scope}
		
	\end{tikzpicture}
	\caption{The prime split graphs $R_{311}, R_{331}$ and their complements.}
	\label{split.primos.deg=4.fig.2}
\end{figure}

\begin{figure}[h]
	\centering
	\begin{tikzpicture}[scale=0.9, 
		every node/.style={draw, circle, fill=white, inner sep=0pt, minimum size=6.5pt},
		]
		
		\begin{scope}[shift={(0,3.2)}]
			\node (T)  at (0, 1)      {};
			\node (BL) at (-0.55, 0)  {};
			\node (BR) at ( 0.55, 0)  {};
			\node (p1) at ( 1.4, 1.3) {}; 
			\node (p2) at ( 1.9, 0.2) {}; 
			\node (p3) at ( 0, -1.1)  {}; 
			\draw (T)--(BL); \draw (T)--(BR); \draw (BL)--(BR);
			\draw (T)--(p1);
			\draw (T)--(p2);
			\draw (BL)--(p3);
			\draw (BR)--(p3);
			\node[draw=none] at (1.3, -0.9) {$R_{211}$};
		\end{scope}
		
		\begin{scope}[shift={(6.6,2.2)}]
			\node (ul) at (0, 1) {};
			\node (ur) at (1, 1) {};
			\node (ll) at (0, 0) {};
			\node (lr) at (1, 0) {};
			\node (w1) at (-1.3, 1.8) {}; 
			\node (w2) at ( 2.3, 1.8) {}; 
			\node (w3) at (-1.3, 0)   {}; 
			\draw (ul)--(ur); \draw (ur)--(lr); \draw (lr)--(ll); \draw (ll)--(ul);
			\draw (ul)--(lr); \draw (ur)--(ll);
			\draw (w1)--(ul); \draw (w1)--(ur);
			\draw (w2)--(ur); \draw (w2)--(lr);
			\draw (w3)--(ll);
			\draw (w1)--(ll);
			\node[draw=none] at (2.3, 0.1) {$R_{321}$};
		\end{scope}
		
		\begin{scope}[shift={(0,0)}]
			\node (A) at (0, 1)      {};
			\node (B) at (-0.55, 0)  {};
			\node (C) at ( 0.55, 0)  {};
			\node (q1) at ( 1.4, 1)  {};
			\node (q2) at ( 1.7, 0)  {};
			\node (q3) at (0,-0.9){};
			\draw (A)--(B); \draw (B)--(C); \draw (A)--(C);
			\draw (A)--(q1);
			\draw (A)--(q2);
			\draw (C)--(q1);
			\draw (C)--(q2);
			\draw (B)--(q3);
			\node[draw=none] at (1.4, -0.9) {$\overline{R_{211}}$};
		\end{scope}
		
		\begin{scope}[shift={(6.6,0)}]
			\node (A) at (0, 1)       {};
			\node (B) at (-0.55, 0)   {};
			\node (C) at ( 0.55, 0)   {};
			\node (q1) at (-1.3, 1)   {};
			\node (q2) at ( 1.3, 1)   {};
			\node (q3) at ( 1.7, 0)   {};
			\node (q4) at ( 0, -1.1)  {};
			\draw (A)--(B); \draw (B)--(C); \draw (A)--(C);
			\draw (A)--(q2);
			\draw (B)--(q1);
			\draw (C)--(q3);
			\draw (B)--(q4);
			\draw (C)--(q4);
			\draw (C)--(q2);
			\node[draw=none] at (2.4, -0.9) {$\overline{R_{321}}$};
		\end{scope}
		
	\end{tikzpicture}
	\caption{The prime split graphs $R_{211}, R_{321}$ and their complements.}
	\label{split.primos.deg=4.fig.3}
\end{figure}
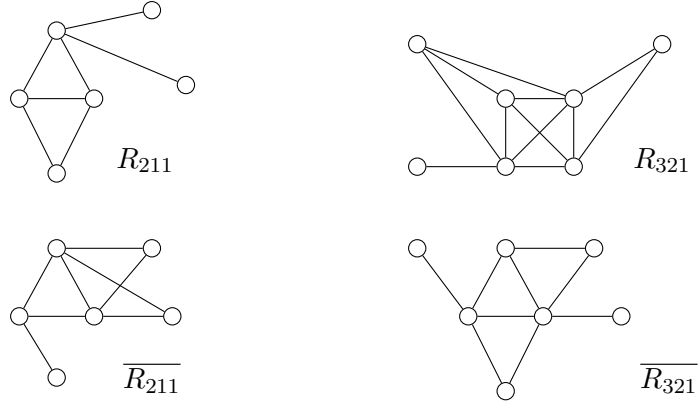
\begin{theorem}
	\label{clasificacion.split.primos.deg=4}
	The 12 graphs in \Cref{split.primos.deg=4.fig.1,split.primos.deg=4.fig.2,split.primos.deg=4.fig.3} are all the prime split graphs of degree 4.
\end{theorem}

\begin{proof}
	It follows from the previous discussion.
\end{proof}		


\section{Prime non-split graphs of degree 4}\label{sec_deg=4.no.split}

For $j\in[3]$, let $B_j$ be the graph with vertex set $\{v,c_0,c_1,c_2,c_3\}$ in which $c_0c_1c_2c_3c_0$ is an induced $C_4$ and $N_{B_j}(v)=\{c_0,\dots,c_{j-1}\}$. 
These six graphs are displayed in \Cref{no.split.primos.deg=4.fig}. In this section we show that they are all the prime non-split graphs of degree 4.

\begin{figure}[ht]
	\centering
	\begin{tikzpicture}[scale=0.9,
		every node/.style={draw, circle, fill=white, inner sep=0pt,
			minimum size=6.5pt},
		]
		
		\begin{scope}[shift={(0,0)}]
			\node (c0) at (0,1) {};
			\node (c1) at (1,1) {};
			\node (c2) at (1,0) {};
			\node (c3) at (0,0) {};
			\node (v)  at (-0.9,1.8) {};
			\draw (c0)--(c1); \draw (c1)--(c2); \draw (c2)--(c3); \draw (c3)--(c0);
			\draw (v)--(c0);
			\node[draw=none] at (0.5,-0.85) {$B_1$};
		\end{scope}
		
		\begin{scope}[shift={(4,0)}]
			\node (c0) at (0,1) {};
			\node (c1) at (1,1) {};
			\node (c2) at (1,0) {};
			\node (c3) at (0,0) {};
			\node (v)  at (0.5,1.9) {};
			\draw (c0)--(c1); \draw (c1)--(c2); \draw (c2)--(c3); \draw (c3)--(c0);
			\draw (v)--(c0); \draw (v)--(c1);
			\node[draw=none] at (0.5,-0.85) {$B_2$};
		\end{scope}
		
		\begin{scope}[shift={(8,0)}]
			\node (c0) at (0,1) {};
			\node (c1) at (1,1) {};
			\node (c2) at (1,0) {};
			\node (c3) at (0,0) {};
			\node (v)  at (1.9,1.8) {};
			\draw (c0)--(c1); \draw (c1)--(c2); \draw (c2)--(c3); \draw (c3)--(c0);
			\draw (v)--(c0); \draw (v)--(c1); \draw (v)--(c2);
			\node[draw=none] at (0.5,-0.85) {$B_3$};
		\end{scope}
		
		\begin{scope}[shift={(0,-3)}]
			\node (a) at (-0.6,1.0) {};
			\node (b) at (-0.6,0.0) {};
			\node (c) at (0.27,0.5) {};
			\node (d) at (1.14,0.5) {};
			\node (e) at (2.01,0.5) {};
			\draw (a)--(b); \draw (a)--(c); \draw (b)--(c);
			\draw (c)--(d); \draw (d)--(e);
			\node[draw=none] at (0.5,-0.6) {$\overline{B_1}$};
		\end{scope}
		
		\begin{scope}[shift={(4,-3)}]
			\node (p1) at (-0.9,0.5) {};
			\node (p2) at (-0.2,0.5) {};
			\node (p3) at ( 0.5,0.5) {};
			\node (p4) at ( 1.2,0.5) {};
			\node (p5) at ( 1.9,0.5) {};
			\draw (p1)--(p2); \draw (p2)--(p3); \draw (p3)--(p4); \draw (p4)--(p5);
			\node[draw=none] at (0.5,-0.6) {$\overline{B_2}$};
		\end{scope}
		
		\begin{scope}[shift={(8,-3)}]
			\node (q1) at (-0.2,1.0) {};
			\node (q2) at ( 0.5,1.0) {};
			\node (q3) at ( 1.2,1.0) {};
			\node (r1) at ( 0.15,0.0) {};
			\node (r2) at ( 0.85,0.0) {};
			\draw (q1)--(q2); \draw (q2)--(q3); \draw (r1)--(r2);
			\node[draw=none] at (0.5,-0.6) {$\overline{B_3}$};
		\end{scope}
		
	\end{tikzpicture}
	\caption{The six non-split prime graphs of degree 4.}
	\label{no.split.primos.deg=4.fig}
\end{figure}
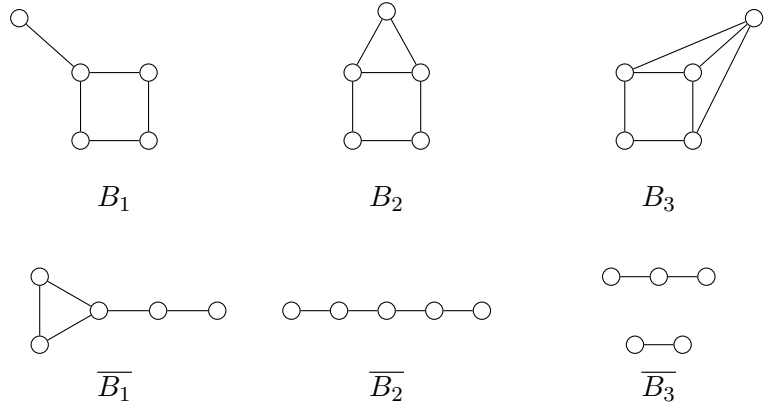

\begin{theorem}
	\label{clasificacion.no.split.primos.deg=4}
	The graphs $B_1,B_2,B_3$ and their complements (see \Cref{no.split.primos.deg=4.fig}) are all the non-split prime graphs of degree $4$.
\end{theorem}

\begin{proof}
		Fix $j\in[3]$ and write $C=B_j[c_0,c_1,c_2,c_3]$. As $|B_j|=5$, every
	member of $\mathcal{Q}_{B_j}$ either equals $C$ or contains $v$; hence
	$\mathcal{Q}_{B_j}=\{C\}\cup\mathcal{H}_v$, so $\deg(B_j)=2+2=4$ by
	\Cref{conteo.local.C4} and \eqref{deg.como.suma}, and $B_j$ is active.
	If $B_j=S\circ H$ with $V(S),V(H)\neq\varnothing$, then $S$ and $H$
	are active by \Cref{hechos.previos}, so $|S|,|H|\geq 4$ by
	\Cref{|G^*|} and $|B_j|\geq 8$, which is absurd; hence $B_j$ is prime.
	Moreover $B_j$ is not split, because $C\approx C_4$. Since
	complementation preserves the degree, primality and non-splitness,
	the same holds for $\overline{B_1},\overline{B_2},\overline{B_3}$.
	Finally, these six graphs are clearly pairwise non-isomorphic.
	
	Conversely, let $G$ be a prime graph with $\deg(G)=4$ which is not
	split. By \Cref{5>degG=|{P_4}|.implica.G.split} we have
	$\deg(G)\neq|\mathcal{Q}_G(P_4)|$, so $\mathcal{Q}_G(C_4)\cup \mathcal{Q}_G(2K_2)\neq\varnothing$ by
	\Cref{degreeofG}. Complementation preserves the degree, primality and
	non-splitness, and interchanges induced $C_4$'s with induced
	$2K_2$'s; as the six graphs of the statement form three complementary
	pairs, we may assume $\mathcal{Q}_G(C_4)\neq\varnothing$ and it suffices to
	prove that $G\approx B_j$ for some $j\in[3]$. Fix an induced $C=c_0c_1c_2c_3c_0\approx C_4$ in
	$G$ and put $M=V(C)$.
	
		By \Cref{C4.distinguidores}, at most one vertex $v_0\in V(G)-M$
	distinguishes $M$, and $N_{v_0}\cap M$ is not a pair of opposite
	vertices of $C$; by \Cref{conteo.local.C4},
	$\sum_{H\in\mathcal{H}_{v_0}}\deg(H)=2$.
	
	Suppose first that no such $v_0$ exists. Then no vertex of $V(G)-M$
	distinguishes $M$, so $M$ is a module of $G$; moreover $G\neq C$
	because $\deg(C)=2$. This contradicts \Cref{C4.modulo.deg>=6}. Hence $v_0$ exists, and \eqref{conteo.global.C4} becomes an equality,
	so $\mathcal{Q}_G=\{C\}\cup\mathcal{H}_{v_0}$. In particular every member
	of $\mathcal{Q}_G$ has all its vertices in $\{v_0\}\cup M$, and since $G$
	is active we get $V(G)=\{v_0\}\cup M$. Finally, after a rotation of
	$C$ we may assume $N_{v_0}\cap M=\{c_0,\dots,c_{j-1}\}$ for some
	$j\in[3]$, because $N_{v_0}\cap M$ is neither empty, nor $M$, nor a
	pair of opposite vertices of $C$. Therefore $G\approx B_j$.
\end{proof}

\begin{corollary}
	\label{clasificacion.primos.deg=4}
	There are exactly 18 prime graphs of degree $4$: the 12
	split graphs of
	\Cref{split.primos.deg=4.fig.1,split.primos.deg=4.fig.2,split.primos.deg=4.fig.3},
	together with the 6 non-split graphs of \Cref{no.split.primos.deg=4.fig}.
\end{corollary}


\section{Consequences for realization graphs}
\label{sec_realization}

We now apply the previous classifications to unigraphs and realization
graphs; if $X$ has degree sequence $d$, we also write $\mathcal{G}(X)$ for
$\mathcal{G}(d)$. Following \cite{tyshkevich2000decomposition}, a degree
sequence is \emph{unigraphical} if all its realizations are isomorphic,
and a \emph{unigraph} is a realization of a unigraphical sequence;
following \cite{barrus2012switches}, a graph is a \emph{hereditary
	unigraph} if all its induced subgraphs are unigraphs, and it is
\emph{$\mathcal{F}$-free} if none of its induced subgraphs lies in the
family $\mathcal{F}$. Let
\[
\mathcal{F}=\bigl\{H,\overline{H}:\ H\in\{P_5,\ K_2\,\dot\cup\,K_3,\ B_1,\ 2P_3,\ K_2\,\dot\cup\,P_4,\ K_2\,\dot\cup\,C_4,\ \overline{T_{221}},\ R_{211}\}\bigr\}.
\]
In \cite{barrus2012switches} $\overline{T_{221}}$ and $R_{211}$ are called $R$ and
$S$ (see Figure~6 of
\cite{barrus2012switches}).

\begin{theorem}[\cite{barrus2012switches}, Theorem 4.2]
	\label{barrus.hereditarios}
	A graph is a hereditary unigraph if and only if it is $\mathcal{F}$-free.
\end{theorem}

If $H\prec G$, then $\mathcal{Q}_H\subseteq\mathcal{Q}_G$, so
\eqref{deg.como.suma} gives
\begin{equation}
	\label{deg.monotono}
	H\prec G\ \Longrightarrow\ \deg(H)\leq\deg(G).
\end{equation}

\begin{theorem}
	\label{grado<=3.unigrafo}
	Let $\deg(G)\leq 3$. Then $G$ is a unigraph if and only if it is a
	hereditary unigraph, if and only if it contains neither $T_{221}$ nor
	$\overline{T_{221}}$ as an induced subgraph; and otherwise some
	2-switch $\tau$ on $G$ satisfies $\deg(\tau(G))\geq 4$. In particular,
	every graph of degree at most $2$ is a hereditary unigraph.
\end{theorem}

\begin{proof}
	As $\deg(\overline{H})=\deg(H)$ for every $H$ \cite{pastine20252}, the
	members of $\mathcal{F}$ other than $T_{221},\overline{T_{221}}$ have
	degree at least $4$. Hence, if $T_{221},\overline{T_{221}}\not\prec G$,
	then $G$ is $\mathcal{F}$-free by \eqref{deg.monotono}, so it is a
	hereditary unigraph by \Cref{barrus.hereditarios}, and in particular a
	unigraph; the same argument proves the last assertion, because
	\eqref{deg.monotono} forbids $T_{221},\overline{T_{221}}$ inside a
	graph of degree at most $2$.
	
	Assume now $T_{221}\prec G$ or $\overline{T_{221}}\prec G$. A 2-switch
	on $\overline{G}$ is a 2-switch on $G$ on the same 4 vertices, the
	resulting graphs being complementary, so we may assume
	$T_{221}\prec G$. It is easy to check that there is a 2-switch $\tau$ on $T_{221}$ such that $\tau(T_{221})\approx \overline{R_{211}}$. Since $\tau(T_{221})\prec\tau(G)$, we have 
	$\deg(\tau(G))\geq 4$ by
	\eqref{deg.monotono}, and since $\tau(G)$ and $G$ share their degree
	sequence, $G$ is not a unigraph.
\end{proof}

\begin{corollary}
	\label{G(d).transitivo}
	Let $X$ realize $d$, with $\deg(X)=k\leq 3$. Then $\mathcal{G}(d)$ is regular if and
	only if $T_{221},\overline{T_{221}}\not\prec X$, and in that case $d$ is unigraphical,
	$\mathcal{G}(d)$ is vertex-transitive and $k$-regular, and $\tau(X)\approx X$ for every
	2-switch $\tau$ on $X$.
\end{corollary}

\begin{proof}
	By \Cref{grado<=3.unigrafo}, $d$ is unigraphical. Let $\Gamma$ be the
	group of permutations $\pi$ of $V(X)$ with $d_{\pi(v)}=d_v$ for all
	$v$. Relabeling by $\pi\in\Gamma$ sends realizations of $d$ to
	realizations of $d$ and preserves the relation of differing by a
	2-switch, so $\Gamma$ acts on $\mathcal{G}(d)$ by automorphisms; and
	any isomorphism between realizations of $d$ preserves degrees, hence
	lies in $\Gamma$. So $\Gamma$ is transitive, $\mathcal{G}(d)$ is
	vertex-transitive, and it is $k$-regular because $X$ has degree $k$ in
	it. The last assertion follows from \Cref{grado<=3.unigrafo}.
\end{proof}

Thus $\mathcal{G}(T_{221})\approx K_{1,2,2}$: four of its vertices are
isomorphic to $T_{221}$ and have degree $3$, while the fifth is
isomorphic to $\overline{R_{211}}$ and has degree $4$; so $T_{221}$ and
$\overline{R_{211}}$ are the two realizations of $d(T_{221})$ up to
isomorphism. This is the first instance of a phenomenon occurring in
every degree $k\geq 3$.

\begin{proposition}
	\label{familia.Gk}
	For every $k\geq 3$ there are a prime graph $G_k$ with $\deg(G_k)=k$
	and a 2-switch $\tau$ on $G_k$ with $\deg(\tau(G_k))=2k-2$; in
	particular, $\mathcal{G}(G_k)$ is not regular.
\end{proposition}

\begin{proof}
	Let $G_k$ be the split graph with clique $K=\{x,y_1,\dots,y_{k-1}\}$,
	independent set $I=\{a,b,c\}$, and $N_a=\{x\}$,
	$N_b=\{y_1,\dots,y_{k-1}\}$, $N_c=\{y_1\}$. By
	\Cref{prop.basicas.sigma_uv}, $\sigma_{ab}=k-1$, $\sigma_{ac}=1$ and
	$\sigma_{bc}=0$, so $\deg(G_k)=k$ and $\Phi(G_k)=cab$ ignoring
	multiplicities. Every vertex of $K\cup\{a,b\}$ lies in an induced $P_4$
	$axy_ib$, and $c$ in the induced $P_4$ $axy_1c$; hence $G_k$ is active,
	and prime by \Cref{S.primo.iff.Phi(S).conexo}. Let $\tau$ replace
	$ax,by_1$ with $ay_1,bx$. In $\tau(G_k)$ we have $N_a=N_c=\{y_1\}$ and
	$N_b=\{x,y_2,\dots,y_{k-1}\}$, so $\sigma_{ab}=k-1=\sigma_{cb}$ and
	$\sigma_{ac}=0$, whence $\deg(\tau(G_k))=2k-2\neq k$.
\end{proof}

Note that $G_3\approx\overline{T_{221}}$ and $G_4\approx R_{311}$ (see
\Cref{grafos.primos.deg=3,split.primos.deg=4.fig.2}). In \Cref{R.k-1.1.1}, $G_5$ and $\tau(G_5)$ are shown.

\begin{figure}[h]
	\centering
	\begin{tikzpicture}[scale=0.8,
			every node/.style={draw, circle, fill=white, inner sep=0pt, minimum size=6.5pt},
			lbl/.style={draw=none, fill=none},
			sw/.style={thick},
			]
			
			\begin{scope}[shift={(0,0)}]
					\node (A1)[fill=gray] at (0,2.7) {};
					\node (A2)[fill=gray] at (150:1.4) {};
					\node (A3) at (110:1.4) {};
					\node (A4) at (70:1.4) {};
					\node (A5) at (30:1.4) {};
					\foreach \i in {1,...,5}
					\foreach \j in {1,...,5}
					{\ifnum\i<\j \draw (A\i)--(A\j);\fi}
					\node (Aa)[fill=gray] at (-0.85,2.7) {};
					\node (Ab)[fill=gray] at (0,0) {};
					\node (Ac) at (-2.5,1.15) {};
					\draw (Ac)--(A2);
					\draw[sw] (Aa)--(A1);
					\draw[sw] (Ab)--(A2);
					\draw (Ab)--(A3); \draw (Ab)--(A4); \draw (Ab)--(A5);
					\node[lbl] at (-1.3,2.7) {$a$};
					\node[lbl] at (0,-0.5) {$b$};
					\node[lbl] at (-2.98,1.15) {$c$};
					\node[lbl] at (0.38,2.78) {$x$};
					\node[lbl] at (-1.5,0.4) {$y_1$};
					\node[lbl] at (-1.5,-0.5) {$G_5$};
				\end{scope}
			
			\begin{scope}[shift={(6.6,0)}]
					\node (B1)[fill=gray] at (0,2.7) {};
					\node (B2)[fill=gray] at (150:1.4) {};
					\node (B3) at (110:1.4) {};
					\node (B4) at (70:1.4) {};
					\node (B5) at (30:1.4) {};
					\foreach \i in {1,...,5}
					\foreach \j in {1,...,5}
					{\ifnum\i<\j \draw (B\i)--(B\j);\fi}
					\node (Ba)[fill=gray] at (-0.85,2.7) {};
					\node (Bb)[fill=gray] at (0,0) {};
					\node (Bc) at (-2.5,1.15) {};
					\draw (Bc)--(B2);
					\draw[sw] (Ba)--(B2);
					\draw[sw] (Bb) to (B1);
					\draw (Bb)--(B3); \draw (Bb)--(B4); \draw (Bb)--(B5);
					\node[lbl] at (-1.3,2.7) {$a$};
					\node[lbl] at (0,-0.5) {$b$};
					\node[lbl] at (-2.98,1.15) {$c$};
					\node[lbl] at (0.38,2.78) {$x$};
					\node[lbl] at (-1.5,0.4) {$y_1$};
					\node[lbl] at (1.5,-0.5) {$\tau(G_5)$};
				\end{scope}
			
			\draw[-{Latex[length=2.2mm]}, thick] (1.85,1.35) -- (3.25,1.35);
			\node[lbl] at (2.55,1.72) {$\tau$};

		\end{tikzpicture}
	\caption{The prime split graph $G_5$ and
			its image under the 2-switch $\tau$, which replaces $ax,by_1$ with
			$ay_1,bx$.}
	\label{R.k-1.1.1}
\end{figure}
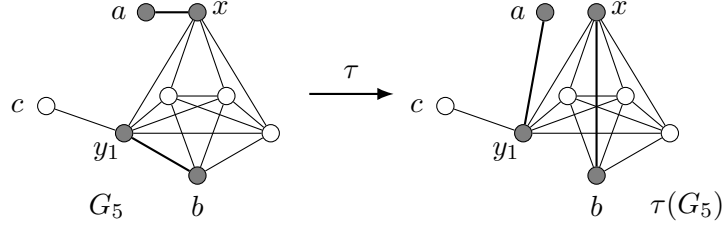

We now turn to the realization graphs themselves. The first result below
reduces their study to prime graphs; in \cite{barrus2016realization} it is
stated for the canonical decomposition of $d$, whose components are the
degree sequences of $G_n,\dots,G_1$.

\begin{theorem}[\cite{barrus2016realization}, Theorem 6]
	\label{G.producto}
	If $G_n\circ\cdots\circ G_1$ is the Tyshkevich decomposition of $X$,
	then
	$\mathcal{G}(X)\approx\mathcal{G}(G_n)\,\square\cdots\square\,\mathcal{G}(G_1)$.
\end{theorem}

\begin{theorem}[\cite{barrus2023cliques}, page 2]
	\label{dual.spaces.iso}
	$\mathcal{G}(d)\approx\mathcal{G}(\overline{d})$.
\end{theorem}

The classification obtained in the previous sections determines the
realization graph of every prime graph of degree at most $4$; by
\Cref{dual.spaces.iso}, only one sequence in each complementary pair has
to be treated.

\begin{theorem}
	\label{realization.graphs.deg<5}
	Let $X$ be a prime graph with $\deg(X)\leq 4$. Then $\mathcal{G}(X)$
	is isomorphic to one of the ten graphs of
	\Cref{realization.graphs.fig}, according to the following table.
	\begin{center}
		\begin{tabular}{ll}
			\hline
			$\mathcal{G}(X)$ & $X$ \\
			\hline
			$K_2$ & $P_4$\\
			$K_3$ & $C_4$, $2K_2$, $D_{2,1}$, $\overline{D_{2,1}}$\\
			$K_4$ & $D_{3,1}$, $\overline{D_{3,1}}$\\
			$K_{3,3}$ & $T_{111}$, $\overline{T_{111}}$\\
			$K_{1,2,2}$ & $T_{221}$, $\overline{T_{221}}$, $R_{211}$, $\overline{R_{211}}$\\
			$K_5$ & $D_{4,1}$, $\overline{D_{4,1}}$\\
			$K_{2,2,2}$ & $D_{2,2}$, $\overline{D_{2,2}}$, $B_1$, $\overline{B_1}$, $B_3$, $\overline{B_3}$\\
			$K_{1,3,3}$ & $B_2$, $\overline{B_2}$\\
			$\mathcal{G}_1$ & $R_{311}$, $\overline{R_{311}}$, $R_{331}$, $\overline{R_{331}}$\\
			$\mathcal{G}_2$ & $R_{321}$, $\overline{R_{321}}$\\
			\hline
		\end{tabular}
	\end{center}
\end{theorem}

\begin{proof}
	By
	\Cref{clasificacion.activo_deg=2,clasificacion.activo_deg=3,clasificacion.split.primos.deg=4,clasificacion.no.split.primos.deg=4}
	there are exactly $29$ prime graphs of degree at most $4$, and they
	realize $26$ distinct degree sequences: the only coincidences are
	$d(T_{221})=d(\overline{R_{211}})$, $d(\overline{T_{221}})=d(R_{211})$
	and $d(B_1)=d(\overline{B_1})$. For each of these sequences, listing
	the realizations of $d$ and the 2-switches between them is a finite
	computation, which yields the table.
\end{proof}

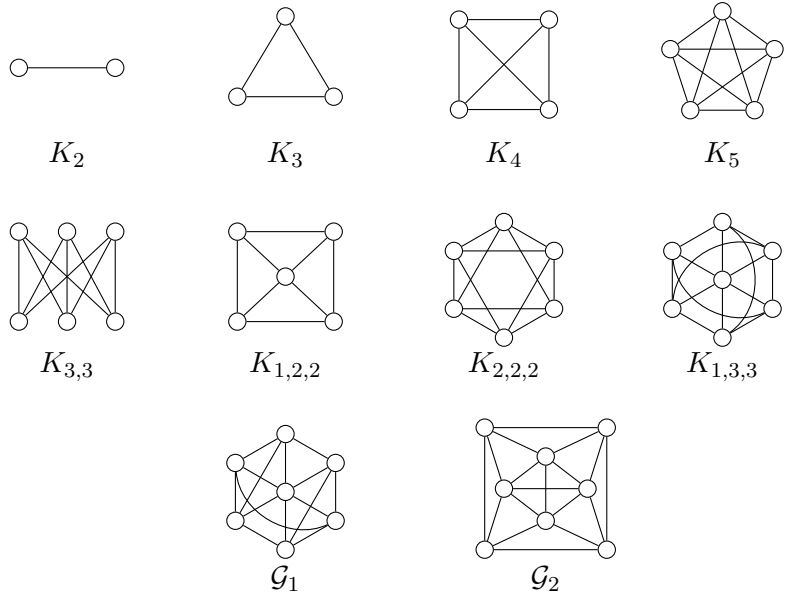
\begin{figure}[ht]
	\centering
	\begin{tikzpicture}[scale=0.85,
		every node/.style={draw, circle, fill=white, inner sep=0pt, minimum size=6.5pt},
		nm/.style={draw=none, fill=none},
		]
		
		\begin{scope}[shift={(0,0)}]
			\node (a) at (0.15,0.9) {};
			\node (b) at (1.65,0.9) {};
			\draw (a)--(b);
			\node[nm] at (0.9,-0.45) {$K_2$};
		\end{scope}
		
		\begin{scope}[shift={(3.4,0)}]
			\node (a) at (0.9,1.7) {};
			\node (b) at (0.15,0.45) {};
			\node (c) at (1.65,0.45) {};
			\draw (a)--(b)--(c)--(a);
			\node[nm] at (0.9,-0.45) {$K_3$};
		\end{scope}
		
		\begin{scope}[shift={(6.8,0)}]
			\node (a) at (0.2,1.65) {};
			\node (b) at (1.6,1.65) {};
			\node (c) at (1.6,0.25) {};
			\node (d) at (0.2,0.25) {};
			\draw (a)--(b)--(c)--(d)--(a); \draw (a)--(c); \draw (b)--(d);
			\node[nm] at (0.9,-0.45) {$K_4$};
		\end{scope}
		
		\begin{scope}[shift={(10.2,0)}]
			\node (a) at (0.90,1.78) {};
			\node (b) at (0.09,1.19) {};
			\node (c) at (0.40,0.24) {};
			\node (d) at (1.40,0.24) {};
			\node (e) at (1.71,1.19) {};
			\draw (a)--(b)--(c)--(d)--(e)--(a);
			\draw (a)--(c); \draw (a)--(d); \draw (b)--(d); \draw (b)--(e); \draw (c)--(e);
			\node[nm] at (0.9,-0.45) {$K_5$};
		\end{scope}
		
		\begin{scope}[shift={(0,-3.3)}]
			\node (t1) at (0.15,1.65) {};
			\node (t2) at (0.90,1.65) {};
			\node (t3) at (1.65,1.65) {};
			\node (b1) at (0.15,0.25) {};
			\node (b2) at (0.90,0.25) {};
			\node (b3) at (1.65,0.25) {};
			\draw (t1)--(b1); \draw (t1)--(b2); \draw (t1)--(b3);
			\draw (t2)--(b1); \draw (t2)--(b2); \draw (t2)--(b3);
			\draw (t3)--(b1); \draw (t3)--(b2); \draw (t3)--(b3);
			\node[nm] at (0.9,-0.45) {$K_{3,3}$};
		\end{scope}
		
		\begin{scope}[shift={(3.4,-3.3)}]
			\node (a) at (0.15,1.65) {};
			\node (b) at (1.65,1.65) {};
			\node (c) at (1.65,0.25) {};
			\node (d) at (0.15,0.25) {};
			\node (o) at (0.90,0.95) {};
			\draw (a)--(b)--(c)--(d)--(a);
			\draw (o)--(a); \draw (o)--(b); \draw (o)--(c); \draw (o)--(d);
			\node[nm] at (0.9,-0.45) {$K_{1,2,2}$};
		\end{scope}
		
		\begin{scope}[shift={(6.8,-3.3)}]
			\node (h1) at (0.90,1.80) {};
			\node (h2) at (0.12,1.35) {};
			\node (h3) at (0.12,0.45) {};
			\node (h4) at (0.90,0.00) {};
			\node (h5) at (1.68,0.45) {};
			\node (h6) at (1.68,1.35) {};
			\draw (h1)--(h2)--(h3)--(h4)--(h5)--(h6)--(h1);
			\draw (h1)--(h3); \draw (h2)--(h4); \draw (h3)--(h5);
			\draw (h4)--(h6); \draw (h5)--(h1); \draw (h6)--(h2);
			\node[nm] at (0.9,-0.45) {$K_{2,2,2}$};
		\end{scope}
		
		\begin{scope}[shift={(10.2,-3.3)}]
			\node (h1) at (0.90,1.80) {};
			\node (h2) at (0.12,1.35) {};
			\node (h3) at (0.12,0.45) {};
			\node (h4) at (0.90,0.00) {};
			\node (h5) at (1.68,0.45) {};
			\node (h6) at (1.68,1.35) {};
			\node (o)  at (0.90,0.90) {};
			\draw (h1)--(h2)--(h3)--(h4)--(h5)--(h6)--(h1);
			\draw (o)--(h1); \draw (o)--(h2); \draw (o)--(h3);
			\draw (o)--(h4); \draw (o)--(h5); \draw (o)--(h6);
			\draw (h1) to[bend left=55] (h4);
			\draw (h2) to[bend right=55] (h5);
			\draw (h3) to[bend left=55] (h6);
			\node[nm] at (0.9,-0.45) {$K_{1,3,3}$};
		\end{scope}
		
		\begin{scope}[shift={(3.4,-6.6)}]
			\node (h1) at (0.90,1.80) {};
			\node (h2) at (0.12,1.35) {};
			\node (h3) at (0.12,0.45) {};
			\node (h4) at (0.90,0.00) {};
			\node (h5) at (1.68,0.45) {};
			\node (h6) at (1.68,1.35) {};
			\node (o)  at (0.90,0.90) {};
			\draw (h1)--(h2)--(h3)--(h4)--(h5)--(h6)--(h1);
			\draw (o)--(h1); \draw (o)--(h2); \draw (o)--(h3);
			\draw (o)--(h4); \draw (o)--(h5); \draw (o)--(h6);
			\draw (h1)--(h3); \draw (h4)--(h6);
			\draw (h2) to[bend right=55] (h5);
			\node[nm] at (0.9,-0.45) {$\mathcal{G}_1$};
		\end{scope}
		
		\begin{scope}[shift={(7.4,-6.6)}]
			\node (c1) at (0.00,1.90) {};
			\node (c2) at (1.90,1.90) {};
			\node (c3) at (1.90,0.00) {};
			\node (c4) at (0.00,0.00) {};
			\node (k1) at (0.95,1.45) {};
			\node (k2) at (1.60,0.95) {};
			\node (k3) at (0.95,0.45) {};
			\node (k4) at (0.30,0.95) {};
			\draw (c1)--(c2)--(c3)--(c4)--(c1);
			\draw (k1)--(k2)--(k3)--(k4)--(k1); \draw (k1)--(k3); \draw (k2)--(k4);
			\draw (c1)--(k1); \draw (c2)--(k1); \draw (c2)--(k2); \draw (c3)--(k2);
			\draw (c3)--(k3); \draw (c4)--(k3); \draw (c4)--(k4); \draw (c1)--(k4);
			\node[nm] at (0.95,-0.45) {$\mathcal{G}_2$};
		\end{scope}
		
	\end{tikzpicture}
	\caption{The ten realization graphs of \Cref{realization.graphs.deg<5}.}
	\label{realization.graphs.fig}
\end{figure}

The four graphs whose realization graph is $K_{1,2,2}$ are precisely the
four members of $\mathcal{F}$ drawn in Fig.~6 of
\cite{barrus2012switches}, and $T_{111}$ is the $3$-net, so
$\mathcal{G}(T_{111})\approx K_{3,3}$ is the transposition graph $T_3$, in
accordance with Example~3 of \cite{barrus2016realization}. In degree $4$
the converse of \Cref{G(d).transitivo} fails: $\mathcal{G}(B_1)\approx
K_{2,2,2}$ is regular, yet $d(B_1)=d(\overline{B_1})$ with
$B_1\not\approx\overline{B_1}$.

\begin{corollary}
	\label{G(d).grado.menor.5}
	If $\mathcal{G}(d)$ has a vertex $X$ of degree $k\in[4]$, then
	$\mathcal{G}(d)$ is the Cartesian product of the realization graphs of
	the prime factors of $X$, each of which is one of the ten graphs of
	\Cref{realization.graphs.fig}. In particular $\mathcal{G}(d)\approx
	K_2$ if $k=1$, and $\mathcal{G}(d)\approx K_3$ or $C_4$ if $k=2$.
\end{corollary}

\begin{proof}
	Let $G_n\circ\cdots\circ G_1$ be the Tyshkevich decomposition of $X$.
	If $|G_i|=1$, then $\mathcal{G}(G_i)\approx K_1$, a neutral factor for
	$\square$; otherwise $A_4(G_i)$ is connected by
	\Cref{indecomp.characterization}, so every vertex of $G_i$ lies in a
	member of $\mathcal{Q}_{G_i}$ and $G_i$ is prime. The prime factors
	$X_1,\dots,X_m$ of $X$ thus satisfy $\sum_i\deg(X_i)=k\leq 4$ by
	\Cref{hechos.previos}, and the first assertion follows from
	\Cref{G.producto,realization.graphs.deg<5}. If $k=1$, then $m=1$ and
	$X_1\approx P_4$; if $k=2$, then either $m=1$ and
	$\mathcal{G}(X_1)\approx K_3$, or $m=2$ and
	$\mathcal{G}(d)\approx K_2\,\square\,K_2\approx C_4$.
\end{proof}

\section*{Acknowledgements}
This work was partially supported by Universidad Nacional de San Luis, grants PI UNSL 03-1326, PROIPRO 03-2923, and PROINI 03-124, and Consejo Nacional de Investigaciones Cient\'ificas y T\'ecnicas grant PIP $11220220$ $100068$CO. 


	
\end{document}